\documentclass[a4paper,12pt,reqno,twoside]{amsart}

\usepackage[utf8]{inputenc} 		
\usepackage[T1]{fontenc} 

\usepackage{amssymb}
\usepackage{amsmath}
\usepackage{graphicx}
\usepackage[colorlinks,citecolor=blue,urlcolor=blue]{hyperref}
\usepackage{cite}
\usepackage[hmargin=2.5cm,vmargin=2.5cm]{geometry}
\usepackage{mathrsfs} 

\newcommand{\be}{\begin{eqnarray}}
\newcommand{\ee}{\end{eqnarray}}

\newcommand{\beq}{\begin{equation}}
\newcommand{\eeq}{\end{equation}}
\newcommand{\beqn}{\begin{equation*}} 
\newcommand{\eeqn}{\end{equation*}}

\newcommand{\round}[1]{\lfloor#1\rfloor}

\DeclareMathOperator{\Var}{\mathrm{Var}}

\newtheorem{thm}{Theorem}[section]
\newtheorem{prop}[thm]{Proposition}
\newtheorem{cor}[thm]{Corollary}
\newtheorem{lem}[thm]{Lemma}
\newtheorem{defn}[thm]{Definition}

\theoremstyle{remark}
\newtheorem{remark}[thm]{Remark}

\newcommand\cB{{\mathcal B}}
\newcommand\cC{{\mathcal C}}

\newcommand\cE{{\mathcal E}}
\newcommand\cF{{\mathcal F}}

\newcommand\cL{{\mathcal L}}
\newcommand\cM{{\mathcal M}}
\newcommand\cN{{\mathcal N}}
\newcommand\cO{{\mathcal O}}

\newcommand\cT{{\mathcal T}}

\newcommand\bE{{\mathbb E}}

\newcommand\bN{{\mathbb N}}

\newcommand\bP{{\mathbb P}}

\newcommand\bR{{\mathbb R}}

\newcommand\bZ{{\mathbb Z}}

\newcommand\rd{{\mathrm d}}

\newcommand{\ve}{\varepsilon}

\def\bfT{\mathbf{T}}

\title[Central limit theorems for time-dependent intermittent maps]{Central limit theorems with a rate of convergence for time-dependent intermittent maps}



\author[Olli Hella]{Olli Hella}
\address[Olli Hella]{
Department of Mathematics and Statistics, P.O.\ Box 68, Fin-00014 University of Helsinki, Finland.}
\email{olli.hella@helsinki.fi}

\author[Juho Lepp\"anen]{Juho Lepp\"anen}
\address[Juho Lepp\"anen]{
 Laboratoire de Probabilités, Statistique et Modélisation (LPSM), CNRS, Sorbonne Université, Université de Paris, 4 Place Jussieu, 75005 Paris, France}
\email{leppanen@lpsm.paris}

\begin{document}
\maketitle

\begin{abstract} We study dynamical systems arising as time-dependent compositions of Pomeau-Manneville-type intermittent maps. We establish central limit theorems for appropriately scaled and centered Birkhoff-like partial sums, with estimates on the rate of convergence. For maps chosen from a certain parameter range, but without additional assumptions on how the maps vary with time,  we obtain a self-norming CLT provided that the variances of the partial sums grow sufficiently fast. When the maps are chosen randomly according to a shift-invariant probability measure, we identify conditions under which the quenched CLT holds, assuming fiberwise centering. Finally, we show a multivariate CLT for intermittent quasistatic systems. Our approach is based on Stein's method of normal approximation.
\end{abstract}

\subsection*{Acknowledgements} The authors would like to thank Mikko Stenlund for helpful comments on preliminary versions of the manuscript. OH was supported by the Academy of Finland, Jane and Aatos Erkko Foundation, and Emil Aaltosen Säätiö. JL was supported by DOMAST (University of Helsinki) and by the European Research Council (ERC) under the European Union's Horizon 2020 research and innovation programme (grant agreement No 787304).

\section{Introduction}

This paper is a continuation of the previous works \cite{Hella_2018,leppanen2017,HellaStenlund_2018} which dealt with statistical properties of time-dependent dynamical systems, such as those described by compositions of the form \(T_n \circ \cdots \circ T_1\) where each  \(T_n : X \to X\)  is a self-map of a probability space $(X, \cB,\mu)$. In \cite{leppanen2017}, the constituent maps $T_n$ were chosen to be Pomeau-Manneville-type intermittent maps $T_{\alpha_n} : [0,1] \to [0,1]$ with parameters $\alpha_n \in (0,1)$. The maps $T_{\alpha_n}$ are expanding but have a common neutral fixed point at the origin, whose influence becomes more significant as $\alpha_n$ increases. A functional correlation bound was shown to facilitate controlling integrals of the form $\int F \circ (T_{\alpha_n} \circ \cdots \circ T_{\alpha_1})_{n = 1}^m \, d \mu$ , where $F$ is not necessarily a product of one-dimensional observables (as in multicorrelation bounds) but a more general functional depending on a finite fragment of the system's trajectory. Such strengthened versions of correlation bounds appear as conditions in general limit theorems for dynamical systems, such as those established in \cite{Pene_2005,Gouezel_2010,HLS_2016}. \\
\indent In  \cite{Hella_2018}, an adaptation of Stein's method \cite{Stein_1972} was developed for normal approximation of time-dependent systems. For a bounded $d$-dimensional observable $f : X \to \bR^d$ the main result of \cite{Hella_2018} gave an upper bound on the distance between the distributions of 
\begin{align*}
W = W(N) = \frac{1}{\sqrt{N}}\sum_{n=0}^{N-1} (f \circ T_n \circ \cdots \circ T_1 - \mu (f \circ T_n \circ \cdots \circ T_1 ))
\end{align*}
and $\cN(0, \text{Cov}_\mu(W))$, under some correlation-decay conditions for the non-stationary process $(f \circ T_n \circ \cdots \circ T_1)_{n\ge1}$. The result was applied to estimate the rate of convergence in the CLT for time-dependent systems composed of smooth uniformly expanding circle maps $T_n : S^1 \to S^1$. Three types of situations were examined, depending on how the circle maps $T_n$ are picked:

\begin{itemize}
\item[(1)] No additional assumptions are made on how the maps are chosen. Then it was shown that the centered partial sums $S = \sum_{n=0}^{N-1} (f \circ T_n \circ \cdots \circ T_1 - \mu (f \circ T_n \circ \cdots \circ T_1 ))$ satisfy a self-norming CLT,
\begin{align}\label{eq:clt_intro}
\frac{S}{\sqrt{\text{Var}_\mu(S)}} \stackrel{D}{\to} \cN(0,1),
\end{align}
with an estimate on the rate of convergence, under the assumption that the growth of $\text{Var}_\mu(S)$ is sufficiently rapid. Here $f : [0,1] \to \bR$ and the density of $\mu$ are assumed to be sufficiently regular. \smallskip
\item[(2)]{The maps are chosen randomly according to a shift-invariant probability measure. In this case a quenched CLT for $W$ was obtained under a strong mixing condition for the selection process. The main result of \cite{HellaStenlund_2018} was used to identify conditions under which the variance $\text{Var}_\mu(W)$ converges almost surely to a positive non-random limit.} \smallskip
\item[(3)]{The dynamics is described by compositions of the form $T_{n,k} \circ \cdots \circ T_{n,1}$, where each $T_{n,k}$ belongs to a triangular array $\{ T_{n,k} \: : \: 0 \le k \le n, \, n \ge 1 \}$. The maps $T_{n,k}$ are assumed to vary slowly with $k$: they approximate a regular curve $\gamma : [0,1] \to \cM$ in the $C^1$-topology, where $\cM$ denotes the space of all admissible circle maps, in the sense that $\lim_{n\to\infty} T_{n,\round{nt}} = \gamma_t$ for all $t \in [0,1]$. The setup is an example of a so-called quasistatic dynamical system (QDS), which is a class of non-stationary systems introduced in \cite{dobbs2016}. QDSs model, among others, slowly transforming thermodynamic processes. The expanding QDS described above was shown to satisfy a multivariate CLT, again augmented by an estimate on the rate of convegence, under some conditions on the regularity of $\gamma$ and the rate of convegence in the above limit.}
\end{itemize}

The purpose of the present paper is to extend the three applications described above to the setting considered in \cite{leppanen2017}, i.e. when the time-dependent system is composed of non-uniformly expanding intermittent maps. In the subsequent sections we state results which give estimates on the convergence rate in the CLT for sequential, random, and quasistatic compositions of these maps. \\
\indent Statistical properties of time-dependent dynamical systems have been studied with increasing focus over the past decade; see for instance \cite{OttStenlundYoung_2009,Stenlund_2011,
StenlundYoungZhang_2013,Stenlund_2014,StenlundSulku_2014,GuptaOttTorok_2013,
MohapatraOtt_2014,Kawan_2015,Kawan_2014,Zhu_etal_2012,
ArnouxFisher_2005,KawanLatushkin_2015, Freitas_2016, rousseau2016}. Self-norming CLTs in the spirit of \eqref{eq:clt_intro} for non-random compositions were obtained in  \cite{bakhtin1994,bakhtin_1994_2} for a class of nearby hyperbolic maps, and in \cite{Nandori_2012, ConzeRaugi_2007, Lothar_1996} for one-dimensional piecewise-expanding maps, where  \cite{Lothar_1996} established also rates of convergence with respect to the Kolmogorov metric. The general operator-theoretic approach of \cite{ConzeRaugi_2007}, which coarsely speaking applies to compositions of maps with quasicompact transfer operators on a suitable Banach, was used in \cite{Haydn_2017} to show almost sure invariance principles also for higher dimensional time-dependent systems.\\
 \indent The study of time-dependent non-uniformly expanding maps was initiated in \cite{aimino2015}, where a statistical memory loss result was established for compositions of intermittent maps $T_{\alpha_n}$ with parameters $0 < \alpha_n \le \beta_* < 1$, but without any other assumptions on how the sequence $(\alpha_n)$ is chosen. The result was a key ingredient in the proof for the functional correlation bound of \cite{leppanen2017} -- one of the main tools of the present manuscript -- and it was also instrumental in \cite{nicol2016}, where the CLT was first examined in this setup. The main result of \cite{nicol2016} shows that, in a certain parameter range, a CLT of the form \eqref{eq:clt_intro} holds for all $C^1$-observables $f : [0,1] \to \bR$, when $\mu$ is the Lebesgue measure, assuming again that $\text{Var}_\mu(S)$ grows sufficiently fast. The last condition was shown to be satisfied by maps $T_{\alpha_n}$ not too far from a fixed map $T_{\alpha}$ for which $f$ is not a co-boundary (in this case the growth of $\text{Var}_\mu(S)$ is linear). Our results imply an upper bound on the rate of convergence in the self-norming CLT with respect to the Wasserstein metric. \\
\indent Quenched limit theorems for random dynamical systems have been established in many articles. We mention \cite{AyyerStenlund_2007,Aimino_2015,Abdelkader_2016,Jianyu_2018, Froyland_2018, Dragicevic_2018} as examples of recent works in this area. In \cite{nicol2016}, a quenched CLT for i.i.d selections of intermittent maps belonging to a finite set of maps was established. In \cite{bahsoun2016,ruziboev2018, bahsoun2017}, quenched correlation bounds with applications to limit theorems for slowly mixing random systems were shown.

\subsection{Notation.} Given a probability space \((X,\cB,\mu)\), and a function \(f: \, X \to \bR^d\), we denote \(\mu(f) = \int_{X} f \, d\mu\). The Lebesgue measure on the unit interval is denoted by \(m\). The coordinate functions of \(f\) are denoted by \(f_{\alpha}\), $\alpha \in\{1,\dots,d\}$, and we set
\beqn
\|f\|_\infty = \max_{1\le\alpha\le d}\|f_\alpha\|_\infty.
\eeqn

\indent We endow \(\bR^d\) with the max-norm \(\Vert x \Vert_{\infty} = \max_{\alpha =1,\ldots ,d} \vert x_{\alpha} \vert \), and for a Lipschitz continuous function \(f \, : \bR \to \bR^{d}\) define 
\begin{align*}
\text{Lip}(f) =  \max_{\alpha = 1,\ldots,d}\sup _{x \neq y} \frac{ | f_{\alpha}(x) - f_{\alpha}(y) | }{| x-y |},
\end{align*}
and \(\Vert f \Vert_{\text{Lip}} = \Vert f \Vert_{\infty} + \text{Lip}(f)\). 

For a function \(B: \, \bR^d \to \bR^{d'}\), we write \(D^kB\) for the $k$th derivative of \(B\), and also denote \(\nabla B = D^1 B\). We define
\begin{align*}
\Vert D^k B \Vert_{\infty} = \max \{ \|\partial_1^{t_1}\cdots\partial_d^{t_d}B_\alpha\|_{\infty}  :  t_1+\cdots+t_d = k,\, 1\le \alpha\le d'  \}.
\end{align*} 
Finally, given two vectors $v,w\in\bR^d$, we write $v\otimes w$ for the $d\times d$ matrix with entries
\beqn
(v\otimes w)_{\alpha\beta} = v_\alpha w_\beta.
\eeqn

Throughout, \(C\) stands for a positive constant whose value may vary from one occurrence to the next. We use \(C(a,b,\ldots)\) to denote positive constants that depend only on the parameters $a,b,\ldots$.

\subsection{Outline of the paper}
In Section \ref{Results} we give some definitions and also all the results of the paper. We also include some of the shorter proofs in Section \ref{Results}. In Section \ref{Stein} and Section \ref{more proofs} we give the rest of the proofs.
\section{Results}\label{Results}

Following \cite{liverani1999}, we define for each \(\alpha \in (0,1)\) the map \(T_{\alpha} : [0,1] \to [0,1]\) by
\begin{align}\label{def:pm}
T_{\alpha }(x) = \begin{cases} x(1+ 2^{\alpha }x^{\alpha}) & \forall x \in [0, 1/2),\\
2x-1 & \forall x \in [1/2,1].
 \end{cases} \hspace{0.5cm}
\end{align} 
Each map has a neutral fixed point at the origin, which together with expansion is responsible for intermittent behavior of the dynamics. Expansion around the origin weakens as \(\alpha\) grows. On the other hand, if \(\alpha \downarrow 0\), the neighborhood in which \(T_{\alpha}' \approx 1\) becomes ever smaller, and at \(\alpha = 0\) we arrive at the uniformly expanding angle-doubling map. \\
\indent Associated to each map \(T_{\alpha}\) is its transfer operator \(\cL_{\alpha} : L^1(m) \to L^1(m)\) defined by
\begin{align*}
\cL_{\alpha} h(x) = \sum_{y \in T_{\alpha}^{-1}\{x\}} \frac{h(y)}{T'_{\alpha}(y)}.
\end{align*}
\indent We denote by \(\hat{\nu}_{\alpha}\) the invariant absolutely continuous probability measure associated to ~\(T_{\alpha}\). It is well-known that the density \(\hat{h}_{\alpha}\) of \(\hat{\nu}_{\alpha}\) is Lipschitz continuous on intervals of the form \([\ve,1]\), where \(\ve > 0\) \cite{young1999,hu2004,liverani1999}. In fact, it follows from \cite{liverani1999} that \(\hat{h}_{\alpha}\) belongs to the convex cone of functions
\beqn
\begin{split}
\cC_*(\alpha) = \{f\in C((0,1])\cap L^1\,:\, & \text{$f\ge 0$, $f$ decreasing,} 
\\
& \text{$x^{\alpha+1}f$ increasing, $f(x)\le 2^{\alpha} (2 + \alpha) x^{-\alpha} m(f)$}\}.
\end{split}
\eeqn
We recall from \cite{liverani1999} that 
\begin{align*}
0 < \alpha \le \beta \hspace{0.2cm}  \Rightarrow \hspace{0.2cm} \cC_*(\alpha) \subset \cC_*(\beta),
\end{align*}
and that
\begin{align*}
0 < \alpha \le \beta\hspace{0.2cm}  \Rightarrow \hspace{0.2cm}  \cL_{\alpha} \cC_*(\beta) \subset  \cC_*(\beta).
\end{align*}
\indent For a single map \(T_{\alpha}\), we denote \(T^{n+1}_{\alpha} = T^n_{\alpha} \circ T_{\alpha}\) for all $n \ge 0$, where \(T^0 = \text{id}_{[0,1]}\). Below we state limit results for time-dependent compositions of the maps $T_{\alpha}$.

\subsection{Time-dependent intermittent maps.} We fix \(\beta_* \in (0,1)\), and call a sequence \((T_{n})_{n \ge 1}\)  of intermittent maps  admissible, if \(0 \le \alpha_n \le \beta_*\) for all \(n \ge 1\). Given such a sequence we denote  \(T_n = T_{\alpha_n}\), and \(\widetilde{T}_n = T_n \circ \cdots \circ T_1\). Let \(\mu\) be a Borel probability measure on \([0,1]\), and let \(f : [0,1] \to \bR^d\) be a function with \(d \ge 1\). For \(N \ge 1\), we denote 
\begin{align*}
W = W(N) = \frac{1}{\sqrt{N}} \sum_{i=0}^{N-1} (f \circ \widetilde{T}_i - \mu(f \circ \widetilde{T}_i)),
\end{align*}
The covariance matrix of \(W\) is denoted by \(\Sigma_N = \mu(W \otimes W)\), and in the case \(d=1\) we set ~\(\sigma_N^2 = \Sigma_N\). \\
\indent Our first theorem, which is proven in Section \ref{Stein}, gives an upper bound on the distance between \(W\) and the $d$-dimensional centered normal distribution $\cN(0,\Sigma_N)$ with covariance matrix~$\Sigma_N$. 

\begin{thm}\label{thm:int_multi} Let \((T_{n})_{n \ge 1}\) be an admissible sequence of intermittent maps. Suppose that the density of the initial measure \(\mu\) belongs to \(\cC_*(\beta_*)\), where \(\beta_* < 1/3\). Let \(N \ge 2\) and let \(f: \, [0,1] \to \bR^d\) be a Lipschitz continuous function such that \(\Sigma_N\) is positive definite. Then, for any three times differentiable function \(h: \, \bR^d \to \bR\) with \(\max_{k=1,2,3}\Vert D^k h \Vert_{\infty} < \infty\),
\begin{align*}
&|\mu(h(W)) - \Phi_{\Sigma_N}(h)| \le C N^{\beta_* - \frac12}(\log N)^{\frac{1}{\beta_*}},
\end{align*}
where \(C > 0\) is a constant independent of \(N\). Here \(\Phi_{\Sigma_N}(h)\) denotes the expectation of \(h\) with respect to $\cN(0,\Sigma_N)$.
\end{thm}

Although Theorem \ref{thm:int_multi} applies in the univariate case \(d=1\), we also have the following complementary result where the smoothness of the test function \(h\) is relaxed to Lipschitz continuity. The expense is that the obtained upper bound depends on the variance \(\sigma_N^2\). Before stating the result, we recall that the Wasserstein distance between two random variables $Y_1$ and $Y_2$  is defined as
\beqn
d_\mathscr{W}(Y_1,Y_2) = \sup_{h\in\mathscr{W}}|\mu(h(Y_1)) - \mu(h(Y_2))|,
\eeqn
where
\beqn
\mathscr{W} = \{h:\bR\to\bR\,:\,|h(x) - h(y)| \le |x-y|\}
\eeqn
is the class of all $1$-Lipschitz functions.

\begin{thm}\label{thm:int} Let \((T_{\alpha_n})_{n \ge 1}\) be an admissible sequence of maps. Let \(Z \sim \cN(0,1)\) be a random variable with normal distribution of mean \(0\) and variance \(1\). Suppose that the density of the initial measure \(\mu\) belongs to \(\cC_*(\beta_*)\), where \(\beta_* < 1/3\). Moreover, let \(N \ge 2\) and let \(f: \, [0,1] \to \bR\) be a Lipschitz continuous function such that \(\sigma_N > 0\). Then,
\begin{align*}
d_{\mathscr{W}}(W,\sigma_{N}Z)
\le C \max\{1, \sigma_{N}^{-2}\} N^{\beta_* - \frac12}(\log N)^{\frac{1}{\beta_*}},
\end{align*}
where \(C > 0\) is a constant independent of \(N\).
\end{thm}
The proof of Theorem \ref{thm:int} is similar to the proof of Theorem \ref{thm:int_multi}. The sketch of the proof of Theorem \ref{thm:int} is given in Section \ref{Stein}. Let \(S = \sqrt{N} W = \sum_{i=0}^{N-1} (f \circ \widetilde{T}_i - \mu(f \circ \widetilde{T}_i))\), and \(\text{Var}(S) = \mu(S^2)\). As an immediate corollary of the above theorem we obtain the following estimate on the rate of convergence in the self-norming univariate CLT: 

\begin{cor}\label{thm:clt_self} Let \(Z\), \(\mu\), \(\beta_*\) and \(f\) be as in Theorem \ref{thm:int}. Assume that $\text{Var}(S) \ge CN^{\ve}$, where $C>0$ and $0\le\ve\le 1$. Then, for any \(N \ge 2\),
\beqn
d_{\mathscr{W}}\left(\frac{S}{\sqrt{\text{Var} (S)}},Z \right)\le  C  N^{1 - \frac32 \ve + \beta_*}(\log N)^{\frac{1}{\beta_*}}.
\eeqn
In particular, if \(\ve >  \tfrac23 (1 + \beta_*)\), then
\begin{align}\label{eq:clt}
\frac{S}{\sqrt{\text{Var}(S)}} \stackrel{D}{\to} \cN (0,1).
\end{align}
\end{cor}

\begin{remark} In the smaller parameter range \(\beta_* < 1/9\), it is seen from Theorem 3.1 of \cite{nicol2016} that \eqref{eq:clt} holds for a weaker lower bound on the variance, namely $\text{Var}(S) \ge CN^{\ve}$  for some \(\ve > 1/2(1-2\beta_*)\) suffices. 
 \end{remark}

\begin{proof}[Proof for Corollary \ref{thm:clt_self}] 
We have $\text{Var}(S)=N\sigma_{N}^{2}$. Therefore 
\beqn
\text{Var}(S)\ge CN^{\epsilon} \Longleftrightarrow \sigma_{N}^{2}\ge CN^{\epsilon-1} \Longleftrightarrow \sigma_{N}^{-2}\le CN^{1-\epsilon}.
\eeqn
By the definition and properties of Wasserstein distance we have
\beqn
d_{\mathscr{W}}\left(\frac{S}{\sqrt{\text{Var}(S)}},Z\right) = d_{\mathscr{W}}\left(\frac{W}{\sigma_{N}},Z\right)
= \sigma_{N}^{-1}d_{\mathscr{W}}\left(W,\sigma_{N}Z\right).
\eeqn
Thus, by Theorem \ref{thm:int},
\beqn
\begin{split}
d_{\mathscr{W}}\left(\frac{S}{\sqrt{\text{Var}(S)}},Z\right)\le  CN^{\frac{1-\ve}{2}} \max\{1, N^{1-\ve}\} N^{\beta_* - \frac12}(\log N)^{\frac{1}{\beta_*}}=C  N^{1 - \frac32 \ve + \beta_*}(\log N)^{\frac{1}{\beta_*}}.
\end{split}
\eeqn
When $\ve > \frac{2}{3} (1 + \beta_*)$, it follows that $C  N^{1 - \frac32 \ve + \beta_*}(\log N)^{\frac{1}{\beta_*}}= o(1)$, which confirms \eqref{eq:clt}.
\end{proof}

We point out that if nothing is assumed on the behavior of \(\sigma_N^2\), our results still imply the following upper bound:

\begin{prop} Let \(Z\), \(\mu\), \(\beta_*\) and \(f\) be as in Theorem \ref{thm:int}. For any \(N \ge 2\), 
\begin{align*}
d_{\mathscr{W}}(W,\sigma_{N}Z)\le C N^{\frac{2\beta_{*}-1}{6}}(\log N)^{\frac{1}{\beta_*}}.
\end{align*}
\end{prop} 

\begin{proof} 
Suppose that $N\ge 2$ is an integer such that $\sigma_{N}>N^{\frac{2\beta_{*}-1}{6}}$. Then~$\max\{1, \sigma_{N}^{-2}\} <N^{\frac{1-2\beta_{*}}{3}}$. Therefore, Theorem \ref{thm:int} implies the bound
\beqn
\begin{split}
d_{\mathscr{W}}(W,\sigma_{N}Z) \le C N^{\frac{1-2\beta_{*}}{3}+\beta_* - \frac12}(\log N)^{\frac{1}{\beta_*}}= C N^{\frac{2\beta_{*}-1}{6}}(\log N)^{\frac{1}{\beta_*}} 
\end{split}
\eeqn
On the other hand, for all random variables $X,Y$ with bounded variances $\sigma_{X}^{2}, \sigma_{Y}^{2}$, respectively, it holds that $d_{\mathscr{W}}(X,Y)\le \sigma_{X}+ \sigma_{Y}$. See e.g. \cite{Hella_2018}. Therefore, for those ~$N\ge 2$ such that $\sigma_{N}\le N^{\frac{2\beta_{*}-1}{6}}$ it holds that $
d_{\mathscr{W}}(W,\sigma_{N}Z) \le  2N^{\frac{2\beta_{*}-1}{6}}$. Thus the claim holds for all integers \(N \ge 2\).
\end{proof}

\subsection{Random dynamical system}\label{subsec:random}
In this subsection we study a setup, where a sequence of Pomeau-Manneville maps is chosen randomly. We show that under some assumptions there exists a limit variance for $W$ and it is the same for almost every random sequence of transformations. \\
\indent Let $(T_{\omega_{i}})_{i=1}^{\infty}$ be a sequence of intermittent maps where $(\omega_i)_{i \ge 1}$ is drawn randomly from the probability space~$(\Omega,\cF,\bP) = ([0,\beta_*]^{\bZ_+},\cE^{\bZ_+},\bP)$. Here $\cE$ is the Borel algebra of $[0,\beta_*]$ and $\bZ_+ = \{1,2,\dots\}$.
 We assume the following about the random dynamical system in question:

\medskip

\noindent{\bf Assumption (RDS)}
\\
i) Each $\omega_{i}\in [0,\beta_*]$.
\\
ii) The law $\bP$ is stationary, i.e., the shift $\tau:\Omega\to\Omega:(\tau(\omega))_i = \omega_{i+1}$ preserves $\bP$.
\\
iii)  The random selection process is strong mixing, satisfying
\beqn
\sup_{i\ge 1} \sup_{A\in \cF_1^i,\, B\in \cF_{i+n}^\infty}|\bP(A \cap B) - \bP(A)\,\bP(B)| \le Cn^{-\gamma}
\eeqn
 for each $n\ge 1$, where $\gamma>0$ and $\cF_1^i$ is a sigma-algebra generated by the projections $\pi_1,...,\pi_i$, $\pi_k(\omega)=\omega_k$, and $\cF_{i+n}^\infty$ is generated by $\pi_{i+n},\pi_{i+n+1}\dots$.

\medskip

 Define $\sigma_{N}^{2}(\omega)=\sigma_{N}^{2}=
\Var_\mu  W(N)$ and $\sigma^{2}=\lim_{N\to\infty}\bE\sigma_{N}^{2}$, when the limit exists. Here $W$ is defined as in the previous subsection except that it now also has $\omega$-dependence. The next theorem gives a quenched convergence result for $W$ that holds for almost every sequence of transformations. 
 \begin{thm} Assume that (RDS) is satisfied with $\beta_* < 1/3$. Then 
 \beqn
 \sigma^{2}= \sum_{k=0}^\infty (2-\delta_{k0})\lim_{i\to\infty}  \bE [\mu(f_i f_{i+k}) - \mu(f_i)\mu(f_{i+k})]
 \eeqn
  is well-defined and non-negative. We have $\sigma>0$ if and only if 
 \beqn
 \sup_{N\ge 1} N\,\bE \mu(W^{2})=\infty.
\eeqn  
  Furthermore if $\sigma> 0$ holds, then for arbitrary $\delta>0$ and almost every $\omega$
 \beqn
d_\mathscr{W}(W(N),\sigma Z) = 
\begin{cases}
O(N^{\beta_* - \frac12}(\log N)^{\frac{1}{\beta_*}}), & \gamma \ge 1,
  \\
  O(N^{\beta_* - \frac12}(\log N)^{\frac{1}{\beta_*}})+O(N^{-\frac{\gamma}{2}} (\log N)^{\frac32+\delta}), &  0<\gamma<1.
\end{cases}
 \eeqn 
 \end{thm}
 \begin{proof}
 We first show that conditions (SA1)--(SA4) in \cite{HellaStenlund_2018} are satisfied:
  
 First we see that Theorem \ref{thm:corr} implies that (SA1) is satisfied with the choice $\eta(0)=\eta(1)=C$ and $\eta(i)=Ci^{-\frac{1}{\beta_*}+1}(\log i)^{\frac{1}{\beta_*}}$ for $n\ge 2$, where $\beta_*\in (0,1/3)$. 
By (RDS) Assumption (SA2) is satisfied with $\alpha(N)=CN^{-\gamma}$.
Assumption (SA3) follows from Theorem 2.6 of \cite{aimino2015}, the corresponding $\eta$ is same as above. By (RDS) $\bP$ is stationary and thus (SA4) holds with a trivial bound $0$ for all $N$.

There exists $\epsilon>0$ such that $\eta(i)\le Ci^{-2-\epsilon}$ for $i\ge 1$. Therefore applying Theorem 4.1 in \cite{HellaStenlund_2018} yields for almost every $\omega$ and any given $\delta>0$.  
\beq
|\sigma_{N}^{2}(\omega)-\sigma^{2}| = 
\begin{cases}
 O(N^{-\frac 12} (\log N)^{\frac32+\delta}), & \gamma > 1,
 \\
  O(N^{-\frac12+\delta}), & \gamma = 1,
  \\
  O(N^{-\frac{\gamma}{2}} (\log N)^{\frac32+\delta}), & 0<\gamma<1.
  \label{eq:sigma^2 bounds}
\end{cases}
\eeq
The formula of $\sigma^{2}$ is also given by the same theorem. The condition for $\sigma>0$ is shown by applying Lemma B.1 of \cite{HellaStenlund_2018}. For this purpose, we need to confirm that Assumption (SA5') in \cite{HellaStenlund_2018} holds.

\textbf{Proof of Assumption (SA5').} The second part of the Assumption (SA5') follows from Theorem 2.6 in \cite{aimino2015}.  For the first part of (SA5') consider the density $h \in \cC_*(\beta_*)$ (the cone corresponding to $T_{\beta_*}$) of $\mu$, which satisfies 
\begin{align*}
h(x) \ge b > 0 \hspace{0.2cm} \text{and} \hspace{0.2cm}  h(x) \le c_{\beta_*} x^{-\beta_*},
\end{align*}
where $b$ and $c_{\beta_*}$ are positive constants depending only on $\beta_*$. Since
\begin{align*}
\frac{\rd\varphi(n,\omega)_*\mu}{\rd\mu}  = \frac{\cL_{\omega_n} \cdots \cL_{\omega_1}h}{h},
\end{align*}
where $\cL_{\omega_n} \cdots \cL_{\omega_1}h \in \cC_*(\beta_*)$, it follows that 
\begin{align*}
\left\|\frac{\rd\varphi(n,\omega)_*\mu}{\rd\mu}\right\|_{L^2(\mu)}^2 &= \int_0^1 \left( \frac{\cL_{\omega_n} \cdots \cL_{\omega_1} h}{h}\right)^2 \, d\mu
= \int_0^1  \frac{(\cL_{\omega_n} \cdots \cL_{\omega_1}\rho)^2}{h} \, d m \\
&\le \frac{c_{\beta_*}^2}{b} \int_0^1 x^{-2\beta_*} \, dx = \frac{c_{\beta_*}^2}{b} \frac{1}{1-2\beta_*},
\end{align*}
where we used $\beta_* < 1/2$. This ends the proof of (SA5').

Assuming that $\sigma>0$ holds, we have $|\sigma_{N}(\omega)-\sigma|=|\sigma_{N}^{2}(\omega)-\sigma^{2}|/|\sigma_{N}(\omega)+\sigma|\le C|\sigma_{N}^{2}(\omega)-\sigma^{2}|$, for every $N\ge 1$.  
 Thus the bounds in the right side of \eqref{eq:sigma^2 bounds} also hold for $|\sigma_{N}(\omega)-\sigma|$. It is easy to show that $d_\mathscr{W}(\sigma_N(\omega) Z,\sigma Z)\le C|\sigma_{N}(\omega)-\sigma|$. Therefore for almost every $\omega$
\beq
 d_\mathscr{W}(\sigma_{N}(\omega)Z,\sigma Z) = 
\begin{cases}
 O(N^{-\frac 12} (\log N)^{\frac32+\delta}), & \gamma > 1,
 \\
  O(N^{-\frac12+\delta}), & \gamma = 1,
  \\
  O(N^{-\frac{\gamma}{2}} (\log N)^{\frac32+\delta}), & 0<\gamma<1.
  \label{eq:wasserst 1}
\end{cases}
\eeq  
 from the assumption $\sigma>0$ and \eqref{eq:sigma^2 bounds} it follows that there exists $N_{0}\in \bN$ such that $\sigma_{N}> 0$ for every $N\ge N_0$ and almost every $\omega$. Now Theorem \ref{thm:int} yields:
\beq
 d_{\mathscr{W}}(W,\sigma_{N}Z)
\le C \max\{1, \sigma_{N}^{-2}\} N^{\beta_* - \frac12}(\log N)^{\frac{1}{\beta_*}}. \label{eq:wasserst 2} 
\eeq
Note that the right side of \eqref{eq:wasserst 2} can be replaced by $C N^{\beta_* - \frac12}(\log N)^{\frac{1}{\beta_*}}$, because we assumed that $\sigma>0$ and thus $\sigma_N>0$ for large values of $N$. The last claim of the theorem now follows by substituting the estimates \eqref{eq:wasserst 1} and \eqref{eq:wasserst 2} in the inequality 
\beqn
 d_\mathscr{W}(W(N),\sigma Z)\le d_\mathscr{W}(W(N),\sigma_{N} Z)+ d_\mathscr{W}(\sigma_{N}Z,\sigma Z).
\eeqn 
 \end{proof}

\subsection{A class of QDSs}\label{qds} Quasistatic dynamical systems were introduced in \cite{dobbs2016} to model situations where the dynamics change very gradually over time due to weak external forces.

\begin{defn}[Discrete time QDS]\label{def:QDS}
Let~$(X,\mathscr{F})$ be a measurable space, $\cM$~a topological space whose elements are measurable self-maps $T:X\to X$, and~$\bfT$ a triangular array of the form
\beqn\label{eq:array}
\bfT = \{T_{n,k}\in\cM\ :\ 0\le k\le n, \ n\ge 1\} .
\eeqn
If there exists a piecewise continuous curve $\gamma:[0,1]\to\cM$ such that\footnote{For any real number \(s \ge 0\), \(\round{s}\) denotes the integer part of \(s\).}
\begin{align}\label{qds_convergence}
\lim_{n\to\infty}T_{n,\round{nt}} = \gamma_t
\end{align}
for all $t$,
we say that $(\bfT, \gamma)$ is a \emph{quasistatic dynamical system (QDS)} with \emph{state space}~$X$ and \emph{system space}~$\cM$. 
\end{defn} 

The limit curve $\gamma$ models the evolution of a slowly transforming system. The evolution of a initial state $x \in X$ under the quasistatic dynamics is described by the array \(\bfT\), separately on each level of the array: \(x_{n,k} = T_{n,k} \circ \cdots \circ T_{n,1} (x) \) is the state of the QDS after \(k \le n\) steps on the \(n\)th level. The aim is to describe the statistical properties of ~\( (x_{n,k})_{0 \le k \le n} \) in the limit \(n \to \infty\). These depend on the curve \(t \mapsto \gamma_t\), which is approximated by \( T_{n,\round{nt}}\) with ever increasing accuracy as \(n\) grows. \\
\indent In the case of a particular system space \(\cM\), each map \(\gamma_t\) typically has an invariant probability measure of special interest. We denote such a designated measure by \(\hat{\mu}_t\). We also fix an initial measure \(\mu\) on the state space $(X,\mathscr{F})$, a bounded measurable function \(f : X \to \bR^d\) with \(d \ge 1\), and introduce the following notations:
\begin{align*}
\bar{f} &= f - \mu(f),\\
\hat{f}_t &= f - \hat{\mu}_t(f), \hspace{0.5cm} 0 \le t \le 1 \\
f^{n,k} &= f \circ T_{n,k} \circ \cdots \circ T_{n,1}, \hspace{0.5cm} 0 \le k \le n,
\end{align*}
where \(f^{n,0} = f\). \\
\indent For each integer \(n \ge 1\), we define the function \(\xi_n : X \times [0,1] \to \bR^d\) by
\begin{align*}
\xi_n(x,t) &= n^{-\frac12}\int^{nt}_0 \bar{f}_{n,\round{s}}(x) \, ds = n^{\frac12} \int_0^t \bar{f}_{n,\round{nr}}(x) \, dr.
\end{align*}
We often hide the \(x\)-dependence here and denote \(\xi_n(t) = \xi_n(x,t)\). Note that if \(nt \in \bN\), then \(\xi_n(x,t) = n^{-\frac12} \sum_{k=0}^{nt-1} \bar{f}_{n,k}(x)\). In other words, \(\xi_n (t)\) is obtained by linearly interpolating scaled and centered time-averages. We denote the covariance matrix of \(\xi_n (t)\) with respect to \(\mu\) by 
\begin{align*}
\Sigma_{n,t} = \mu [\xi_n(t) \otimes \xi_n(t)], \hspace{0.5cm} n \ge 1,\, t \in [0,1],
\end{align*}
and also set
\begin{align*}
\hat{\Sigma}_t(f)=\lim_{m\rightarrow \infty}\hat{\mu}_{t}\left[\frac{1}{\sqrt{m}}\sum_{k=0}^{m-1}\hat{f}_{t}\circ \gamma_{t}^{k}\otimes \frac{1}{\sqrt{m}}\sum_{k=0}^{m-1}\hat{f}_{t}\circ \gamma_{t}^{k}\right],
\end{align*}
given that the limit exists. \\
\indent The previous papers \cite{dobbs2016,Stenlund_2015, leppanen2016, leppanen2017_2, Hella_2018} dealt with statistical properties of QDSs. In \cite{Hella_2018}, the CLT was established for a class of QDSs constructed over uniformly expanding circle maps. Below we give results that extend those of \cite{Hella_2018} to a class of polynomially mixing QDSs. Note that, when the system space \(\cM \) is formed by the intermittent maps $T_{\alpha}$, Theorem \ref{thm:int_multi} guarantees that, given suitable restrictions on the range of \(\gamma\), the distribution of \( \xi_{n}(\lceil nt\rceil/n) = n^{-\frac12}\sum_{k=0}^{\lceil nt\rceil/n-1}\bar{f}_{n,k}\) is close to \(\cN (0,\Sigma_{n,\lceil nt\rceil/n})\) for large \(N\). Hence, to obtain the CLT, it remains to identify \(\lim_{n\to\infty} \Sigma_{n,\lceil nt\rceil/n}\). We now state conditions for an abstract QDS that imply
\begin{align*}
\lim_{n \to \infty} \Sigma_{n,\lceil nt\rceil/n} = \int_0^t \hat{\Sigma}_s(f) \, ds
\end{align*}
with a polynomially decaying error. These conditions are shown to be satisfied by intermittent maps in Section \ref{int_qds_clt}.

\noindent\textbf{Conditions.} Set \(\mathbf{T}' = \mathbf{T} \cup \{ \gamma_t \: : \: t \in [0,1] \}\) and \(\cC = \bigcup_{k=0}^{\infty} C_k\), where
\begin{align*}
C_0 &= \{ \hat{\mu}_t,\,\mu \: : \: t \in [0,1] \},\\
C_{k+1} &= \{ (T)_*\nu \: : \: \nu \in C_k,\, T \in \mathbf{T}' \}.
\end{align*}
Below \(\cT_k\) stands for any \(k\)-composition \(T_k \circ \cdots \circ T_1\) of maps \(T_i \in \mathbf{T}'\). We assume the existence of a constant \(C > 0\), such that the following conditions hold for all bounded functions \(F\) of the form \(F = f_{a} \cdot f_b^q \circ T_k \circ \cdots \circ T_1\) where \(T_i \in \mathbf{T}'\), \(a,b \in \{1,\ldots,d\}\), \(q \in \{0,1\}\):
\begin{itemize}
\item[\textbf{(I)}]{There is \(\varphi > 1\) and \(n_1 \ge 1\), such that for any \(\nu^1, \nu^2 \in \cC\), and any integers \(n ,m \ge 0\) with \(m - n \ge n_1\),
\begin{align*}
|\nu^1(f_{\alpha}^p \circ \cT_n \cdot F \circ \cT_m ) - \nu^1(f_{\alpha}^p \circ \cT_n) \nu^2(  F \circ \cT_m )| \le C(m-n)^{-\varphi},
\end{align*}
whenever \(\alpha \in \{1,\ldots,d\}\) and \(p \in \{0,1\}\).}
\smallskip
\item[\textbf{(II)}] There is \(\psi \in (0,1]\), such that for all integers \(k,m,n \ge 0\) with 
\( k+m \le n\), measures \(\nu \in \cC\), \(s, r_1,\ldots,r_k \in [0,1]\), \(\alpha \in \{1,\ldots,d\}\), and \(p \in \{0,1\}\):
\begin{align*}
&\left|\nu \left[ f_{\alpha}^p \cdot ( F \circ T_{n,m+k}\circ \cdots \circ T_{n,m+1} - F \circ \gamma_{(m+k) / n} \circ \cdots \circ \gamma_{(m+1) / n}) \right] \right| \le C k n^{-\psi},
\end{align*}
and
\begin{align*}
&|\hat{\mu}_s[f_{\alpha}^p \cdot ( F \circ \gamma_s^k - F \circ \gamma_{r_k} \circ \cdots \circ \gamma_{r_1})]| \le C k \max_{1 \le l \le k} |s - r_l|^{\psi}.
\end{align*}
\end{itemize}

Condition (I) is a combination of two conditions: when $p=0$, the bound states a (sufficiently rapid) polynomial memory loss property, and when $\nu^1 = \nu^2$, the bound states a polynomial rate of correlation decay. In the case of intermittent maps, the latter property follows from the former one. Condition (II) on the other hand is a type of perturbation estimate. To clarify its meaning, we might take $p=0$ and $k=1$, and notice that then the two bounds in the condition become
\begin{align*}
&| (T_{n,m+1})_*\nu (  F )  - (\gamma_{(m+1)/n})_*\nu (F)  | \le C n^{-\psi},
\end{align*}
and
\begin{align*}
&| (\gamma_s)_*\hat{\mu}_s( F)   -  (\gamma_r)_*\hat{\mu}_s( F)  | \le C |s - r|^{\psi}.
\end{align*}
Recall that, by \eqref{qds_convergence}, $\lim_{n\to\infty}T_{n,\round{nt}} = \gamma_t$ holds in the space $\cM$ so that the former bound can be viewed as a condition specifying the type and rate of this convergence. The latter bound concerns the regularity of $\gamma$. For intermittent maps we can expand the integral expressions in condition (II) using transfer operators, after which the condition is a direct consequence of the foregoing two bounds. This is verified in Section \ref{int_qds_clt}.\\
\indent Next we state the promised implication of the above conditions regarding the convergence of the covariance matrix. Set
\begin{align*}
\Sigma_t = \int_0^t \hat{\Sigma}_s(f) \, ds.
\end{align*}

\begin{thm}\label{thm:qds_abs_1} Suppose that conditions (I) and (II) hold. Then, given any \(\ve > 0\), 
\begin{align*}
[\Sigma_{n,t}]_{\alpha\beta}-[\Sigma_{t}]_{\alpha\beta} = \cO(n^{\max\{(\psi-\varphi\psi)/(\varphi + \psi + 1)+\epsilon,-1/6\}}),
\end{align*}
for every $\alpha,\beta \in \{1,2,...,d \}$. 
\end{thm}

The following result, which essentially follows from Theorem \ref{thm:qds_abs_1}, is used below to show a CLT for  a QDS composed of intermittent maps.

\begin{thm}\label{thm:qds_abs_2} Let \(h: \, \bR^d \to \bR\) be a Lipschitz continuous function and $t_{0}\in (0,1]$. Suppose that conditions (I) and (II) hold in addition to the following two conditions.

\begin{itemize}
\item[\textbf{(III)}]{ There exists $\zeta \in (0,1)$ such that for every $n\in \bN$ and $t\ge t_{0}$,
\[
 \left|\mu\left[h(\xi_{n}(\lceil nt\rceil/n))\right]- \Phi_{\Sigma_{n,\lceil nt\rceil/n}}(h)\right|  \le Cn^{-\zeta}. 
\]
}
\item[\textbf{(IV)}]{$\hat{f}_{t_0}$ is not a co-boundary for $\gamma_{t_0}$ in any direction.}\footnote{i.e. there is no unit vector $v \in \mathbb{R}^d$ and a function $g_v:X\to\bR$ in $L^2(\mu)$ such that $v\cdot f = g_v - g_v\circ \gamma_{t_0}$.}
\end{itemize}

Then, for every $t \ge t_0$, \(\Sigma_{t}\) is positive definite and 
\beqn
\left|\mu\left[h(\xi_{n}(t))\right]-\Phi_{\Sigma_{t}}(h)\right|\le C n^{\max\{(\psi-\varphi\psi)/(\varphi + \psi + 1)+\epsilon,-1/6,-\zeta\}}
\eeqn
holds. Here $\epsilon>0$ can be chosen arbitrary small, and the constant $C > 0$ depends on $t_0$ but not on $t$.\footnote{Condition (IV) implies the positive definiteness of $\hat{\Sigma}_{t_0}$. If the latter property is required for $t_0=0$, then $\hat{\Sigma}_{t_0}$ is also positive definite for some $t_0 > 0$. However, if we require condition (IV) only for $t_0 = 0$, the constant $C$ in the last upper bound depends also on $t$.}
\end{thm}
The above theorems are proven in Section \ref{more proofs}.
 
\subsection{CLT for the intermittent QDS}\label{int_qds_clt} Recall that a QDS is a pair \((\mathbf{T}, \gamma) \) where \(\mathbf{T} = \{ T_{n,k} \: : \: 0 \le k \le n, \, n \in \bN \} \) is a triangular array of maps in a topological space \(\cM\), and \(\gamma : [0,1] \to \cM \) is a curve such that \(T_{n,\round{nt}} \to \gamma_t\) as \(n \to \infty\). In \cite{leppanen2016}, the following intermittent version of the QDS was introduced.

\begin{defn}[Intermittent QDS]\label{def_int_qds}
Let~$X = [0,1]$ and $\cM = \{T_\alpha\,:\,0 \le\alpha < 1\}$ (equipped, say, with the uniform topology). Next, let
\beqn
\{\alpha_{n,k}\in [0,1)\ :\ 0\le k\le n, \ n\ge 1\}
\eeqn
be a triangular array of parameters and
\beqn
\tau : [0,1] \to [0,1)
\eeqn
a piecewise continuous curve satisfying
\beqn
\lim_{n\to\infty} \alpha_{n,\round{nt}} = \tau_t
\eeqn
for all $t$.
Finally, define $\gamma_t = T_{\tau_t}$ and
\beqn\label{eq:T_nk}
T_{n,k} = T_{\alpha_{n,k}}.
\eeqn
\end{defn}

For clarity we recast some of the definitions introduced in Section \ref{qds} for the intermittent QDS. Given a bounded measurable function $f : [0,1] \to \bR^d$, we denote 
\begin{align*}
f^{n,k} &= f \circ T_{\alpha_{n,k}} \circ \cdots \circ T_{\alpha_{n,1}}, \hspace{0.5cm} 0 \le k \le n.
\end{align*}
Recall that, for each \(t \in [0,1]\), there is a \(T_{\tau_t}\)-invariant probability measure \(\hat{\mu}_t\), namely ~\(\hat{\mu}_t = \hat{\nu}_{\tau_t}\) is the SRB-measure. We fix an initial distribution $\mu$ of $x \in [0,1]$, and denote
\begin{align*}
\bar f = f - \mu(f).
\end{align*}
Then, the fluctuations $\xi_n : [0,1] \times [0,1] \to \bR^d$ are defined by
\begin{align*}
\xi_n(x,t) &= n^{\frac12} \int_0^t \bar{f}_{n,\round{nr}}(x) \, dr.
\end{align*}
In  \cite{leppanen2017_2} it was shown that, with appropriate conditions on the limiting curve $\tau$, the process $(\xi_n)_{n \ge 1}$ converges weakly to a stochastic diffusion process, when each map $x \mapsto \xi_n(x,\cdot)$ is viewed as a random element with values in $C([0,1],\bR^d)$. Below we fix $t \in [0,1]$ and approximate $\xi_n(\cdot,t)$ by a normal distribution of $d$ variables. 

\begin{thm}\label{thm:qds_pm} Let \(f : [0,1] \to \bR^d\) be a Lipschitz continuous function, and let the initial measure \(\mu\) be such that its density is in \(\cC_*(\beta_*)\). Suppose that the limiting curve \(\tau : [0,1] \to [0,1]\) is Hölder-continuous of order \(\eta \in (0,1]\), that \(\tau([0,1]) \subset [0,\beta_*]\) for some \(\beta_* < 1/3\), and that 
\begin{align}\label{tau_as}
\sup_{n\ge 1} n^{\eta}\sup_{t \in [0,1]} | \alpha_{n,\round{nt}} - \tau_t | < \infty.
\end{align}
Suppose there exists \(t_{0}\in (0,1]\) such that $\hat{f}_{t_0}$ is not a co-boundary for $\gamma_{t_0}$ in any direction. Then for all $t \ge t_0$,  \(\Sigma_{t}\) is positive definite, and for all three times differentiable functions \(h : \bR^d \to \bR\) with \( \max_{k=1,2,3} \Vert D^k h \Vert_{\infty} < \infty \),
\begin{align}\label{eq:main_int_qds}
\left|\mu\left[h(\xi_{n}(t))\right]-\Phi_{\Sigma_{t}}(h)\right|\le Cn^{-\theta},
\end{align}
where $C > 0$ is independent of $t$, and 
\begin{align*}
\theta = \frac{1}{\frac{12}{\eta (1-\beta_*)} + 1}.
\end{align*}
 \end{thm}
 
 \begin{proof} The assumptions imply that there is \(\beta_* < 1/3\) and \(n_0 \ge 1\) such that \(\alpha_{n,k} \le \beta_*\) whenever \(n \ge n_0\) and \(0 \le k \le n\). We will, without loss of generality, assume that this holds for \(n_0 = 1\).\\ 
\indent The proof proceeds by verifying conditions (I)-(III) of Theorem \ref{thm:qds_abs_2}. Since  \(\cC_*(\alpha) \subset \cC_*(\beta_*)\) and  \( \cL_{\alpha} \cC_*(\beta_*) \subset  \cC_*(\beta_*)\) hold whenever \(\alpha \le \beta_*\), it suffices to verify conditions (I) and (II) for measures whose densities lie in the cone \(\cC_*(\beta_*)\).
\begin{itemize}
\item[(I)]{Since $f$ is Lipschitz continuous and \(\beta_* < 1/3\), it follows from \cite{leppanen2017} (or from Theorem \ref{thm:corr} below) that condition (I) holds for any \(\varphi > 2\).} \smallskip
\item[(II)]{Let \(\nu\) be a measure with density \(g \in \cC_*(\beta_*)\). We need to show that for some $\psi \in (0,1)$ and $C > 0$,
\begin{align}\label{eq:toshow}
|\nu [f_{\alpha}^p \cdot ( F \circ T_{n,m+k}\circ \cdots \circ T_{n,m+1} - F \circ T_{\tau_{(m+k) / n}} \circ \cdots \circ T_{\tau_{(m+1) / n}})]| \le C k n^{-\psi},
\end{align}
where $F$ is a bounded function, $p \in \{0,1\}$, $\alpha \in \{1,\ldots,d \}$, and $k + m \le n$. For each $0  \le k \le n$ we denote $\cL_{n,k} = \cL_{\alpha_{n,k}}$. Then,
\begin{align*}
&|\nu [f_{\alpha}^p ( F \circ T_{n,m+k}\circ \cdots \circ T_{n,m+1} - F \circ T_{\tau_{(m+k) / n}} \circ \cdots \circ T_{\tau_{(m+1) / n}} )]| \\
&\le \Vert F \Vert_{\infty} \Vert (\cL_{n,m+k} \cdots \cL_{n,m+1} - \cL_{\tau_{(m+k)/n}} \cdots \cL_{\tau_{(m+1)/n}}) g f_{\alpha}^p) \Vert_1 \\
&\le \Vert F \Vert_{\infty} \sum_{l = m+1}^{m+k} \Vert \cL_{n,k+m} \cdots \cL_{n,l+1} (\cL_{n,l} - \cL_{\tau_{l/n}}) \cL_{\tau_{(l-1)/n}} \cdots \cL_{\tau_{(m+1)/n}} g f_{\alpha}^p \Vert_1 \\
&\le k \Vert F \Vert_{\infty} \max_{m+1 \le l \le m+k}  \Vert  (\cL_{n,l} - \cL_{\tau_{l/n}}) \cL_{\tau_{(l-1)/n}} \cdots \cL_{\tau_{(m+1)/n}} g f_{\alpha}^p \Vert_1,
\end{align*}
where the last inequality holds because \(\cL_{\alpha}\) is an \(L^1\)-contraction. By Lemma 2.4 in \cite{leppanen2017}, for each \(l\) with \(m +1 \le l \le m+k\), there are functions \(g_1,g_2 \in \cC_*(\beta_*)\), such that
\begin{align*}
\cL_{\tau_{(l-1)/n}} \cdots \cL_{\tau_{(m+1)/n}} g f_{\alpha}^p = g_1 - g_2,
\end{align*}
and \( \Vert g_i \Vert_1 \le C(\beta_*)(\Vert f \Vert_{\text{Lip}} +1)\) for some constant \(C(\beta_*) > 0\) depending only on ~\(T_{\beta_*}\).  Theorem 5.1 in \cite{leppanen2016} applies to cone functions, and it follows that
\begin{align*}
&\Vert  (\cL_{n,l} - \cL_{\tau_{l/n}}) \cL_{\tau_{(l-1)/n}} \cdots \cL_{\tau_{(m+1)/n}} g f_{\alpha}^p \Vert_1 \\
&\le C(\beta_*) (\Vert f \Vert_{\text{Lip}} + 1) |\alpha_{n,l} - \tau_{l/n}|^{\frac14 (1-\beta_*)},
\end{align*}
where
\begin{align*}
|\alpha_{n,l} - \tau_{l/n}|^{\frac14 (1-\beta_*)} \le C(\tau,\beta_*) n^{- \eta \frac14 (1-\beta_*) }.
\end{align*}
We conclude that \eqref{eq:toshow} holds with $\psi = \eta (1-\beta_*)/4$. Since \(\gamma\) is Hölder continuous of order \(\eta\), an argument similar to the one above shows that also the second bound in condition (II) holds with $\psi = \eta (1-\beta_*)/4$.
} \smallskip
\item[(III)]{By Theorem \ref{thm:int_multi}, condition (III) of Theorem \ref{thm:qds_abs_2} holds with any positive \(\zeta < 1/2 - \beta_*\); in particular with $\zeta = 1/6$.}
\end{itemize}
We have verified conditions (I)-(III). By our assumption, also condition (IV) holds, and now Theorem \ref{thm:qds_abs_2}  implies that
\begin{align*}
\left|\mu\left[h(\xi_{n}(t))\right]-\Phi_{\Sigma_{t}}(h)\right|\le Cn^{-\theta},
\end{align*}
where
\begin{align*}
\theta = \min \left\{ \frac{ \psi (\varphi - 1)}{\varphi + \psi + 1} - \epsilon,\frac16, \zeta \right\},
\end{align*}
and $\epsilon$ can be chosen arbitrarily small. Finally, 
\begin{align*}
\frac{ \psi (\varphi - 1)}{\varphi + \psi + 1} > \frac{\psi}{3 + \psi} = \frac{1}{\frac{12}{\eta (1-\beta_*)} + 1},
\end{align*}
from which the desired bound follows.
\end{proof}

\section{Proofs for Theorems \ref{thm:int_multi} and \ref{thm:int}}\label{Stein}
The proofs for Theorems \ref{thm:int_multi}  and  \ref{thm:int} are based on an adaptation of Stein's method for dynamical systems \cite{HLS_2016}, and in particular its non-stationary version developed in \cite{Hella_2018}, which we next recall.

\subsection{Stein's method} Let \((f^i)_{i=0}^{\infty}\) be a sequence of random vectors with values in \(\bR^d\), given a probability space \((X,\cB,\mu)\). We set
\begin{align*}
W = W(N) = \frac{1}{\sqrt{N}} \sum_{k=0}^{N-1} f^k,
\end{align*}
for all \(N \in \bN_0 = \{0,1,2,\ldots\} \). For \(0 \le K < N\), we define the time window
\begin{align*}
[n]_{N,K} = \{ k \in \bN_0 \cap [0,N-1] \: : \: |k-n| \le K \}
\end{align*} 
around \(n \ge 0\), and
\begin{align*}
W_n = W - \frac{1}{\sqrt{N}} \sum_{k \in [n]_{N,K}} f^k.
\end{align*}
The covariance matrix of \(W\) is denoted by \(\Sigma_N = \mu(W \otimes W)\), and in the case \(d=1\) we let \(\sigma_N^2 = \Sigma_N\).\\
\indent The following theorem was proved in \cite{Hella_2018}, and it shows that under certain correlation decay conditions the distribution of $W$ is close to $\cN(0,\Sigma_N)$.

\begin{thm}\label{thm:main} Assume that the random vectors \(f^i\) have a common upper bound $\|f\|_{\infty}\ge \|f^{i}\|_{\infty}$, for every $i\in \bN_{0}$. Let \(h: \, \bR^d \to \bR\) be three times differentiable with \(\Vert D^k h \Vert_{\infty} < \infty\) for \(1 \le k \le 3\).  Fix integers $N>0$ and $0\le K<N$. Suppose that the following conditions are satisfied:

\begin{itemize}
\item[(A1)]\label{A1} There exist constants $C_2 > 0$ and $C_4 > 0$, and a non-increasing function \(\rho : \, \bN_0 \to \bR_+\) with \(\rho(0) = 1\) and \( \sum_{i=1}^{\infty} i\rho(i) < \infty\), such that for all \(0\le i \le j \le k \le l \),
\begin{align*}
|\mu(\bar{f}^{i}_{\alpha}\bar{f}^{j}_{\beta})| \leq C_2\rho(j-i),
\end{align*}
and
\begin{align*}
&|\mu(\bar{f}^{i}_{\alpha}\bar{f}^{j}_{\beta}\bar{f}^{k}_{\gamma}\bar{f}^{l}_{\delta})| \le C_4\rho(\max\{j-i,l-k\}), \\
&|\mu(\bar{f}^{i}_{\alpha}\bar{f}^{j}_{\beta}\bar{f}^{k}_{\gamma}\bar{f}^{l}_{\delta}) - \mu(\bar{f}^{i}_{\alpha}\bar{f}^{j}_{\beta})\mu(\bar{f}^{k}_{\gamma}\bar{f}^{l}_{\delta})| \le C_4\rho(k-j)
\end{align*}
hold, where $\alpha,\beta,\gamma,\delta\in\{\alpha',\beta'\}$ and $\alpha',\beta'\in\{1,\dots,d\}$.

\smallskip
\item[(A2)]\label{A2} There exists a function \(\tilde\rho : \, \bN_0 \to \bR_+\) such that
\begin{align*}
|\mu( \bar{f}^n \cdot  \nabla h(v + W_n t) ) | \le \tilde\rho(K)
\end{align*}
holds for all $0\le n \le N-1$, $0\le t\le 1$ and $v\in\bR^d$.

\smallskip
\item[(A3)]\label{A3} 
$\Sigma_{N}$
is a positive-definite $d\times d$ matrix.
\end{itemize}
Then
\beq
|\mu(h(W)) - \Phi_{\Sigma_{N}}(h)| \le 
C\left(\frac{K+1}{\sqrt{N}}+\sum_{i=K+1}^{\infty}\rho(i)\right) + \sqrt N\tilde\rho(K),\label{main thm eq1}
\eeq
where $C$ is independent of $N$ and $K$.

\end{thm}

If \(d=1\), we have another result of \cite{Hella_2018} establishing convergence to a normal distribution, where the error is estimated using the Wasserstein distance. Recall that given two random variables $X_1$ and $X_2$, the Wasserstein distance \(d_\mathscr{W}(X_1,X_2)\) between them is defined by
\beqn
d_\mathscr{W}(X_1,X_2) = \sup_{h\in\mathscr{W}}|\mu(h(X_2)) - \mu(h(X_2))|,
\eeqn
where
\beqn
\mathscr{W} = \{h:\bR\to\bR\,:\,|h(x) - h(y)| \le |x-y|\}
\eeqn
is the class of all $1$-Lipschitz functions.

\begin{thm}\label{main thm 1d}
 Let $(X,\cB,\mu)$ be a probability space and $(f^{i})_{i=0}^{\infty}$ a sequence of random variables with common upper bound $\|f\|_{\infty}$. Fix integers $N>0$ and $0\le K<N$. Suppose that the following conditions are satisfied.

\begin{itemize}
\item[(B1)]\label{B1} There exist constants \(C_2,C_4\) and a non-increasing function \(\rho : \, \bN \to \bR\) with \(\rho(0) = 1\), such that for all \(0 \le i \le j \le k \le l\),
\begin{align*}
|\mu(\bar{f}^{i}\bar{f}^{j})| \leq C_2\rho(j-i),
\end{align*}
\begin{align*}
&|\mu(\bar{f}^{i}\bar{f}^{j}\bar{f}^{k}\bar{f}^{l})| \le C_4\rho(\max\{j-i,l-k\}), \\
&|\mu(\bar{f}^{i}\bar{f}^{j}\bar{f}^{k}\bar{f}^{l}) - \mu(\bar{f}^{i}\bar{f}^{j})\mu(\bar{f}^{k}\bar{f}^{l})| \le C_4\rho(k-j).
\end{align*}
\item[(B2)]\label{B2} There exists a function \(\tilde\rho : \, \bN_0 \to \bR_+\) such that, given a differentiable $A:\bR\to\bR$ with $A'$ absolutely continuous and $\max_{0\le k\le 2}\|A^{(k)}\|_\infty \le 1$,
\begin{align*}
|\mu( \bar{f}^n A(W_{n}) ) | \le \tilde\rho(K)
\end{align*}
holds for all $0\le n < N$.
\item[(B3)]\label{B3}
$\sigma_{N}^{2}>0$.
\end{itemize}
Then the Wasserstein distance $d_\mathscr{W}(W,\sigma_{N}Z)$ is bounded from above by 

\beqn
C\left(\max\{\sigma_{N}^{-1},\sigma_{N}^{-2}\}\!\left(\frac{K+1}{\sqrt{N}} + \sum_{i= K+1}^{\infty}  \rho(i) \right) + \max\{1,\sigma_{N}^{-2}\}\sqrt N\tilde\rho(K)\right),
\eeqn
where
$C$
is independent of $N$ and $K$.
\end{thm}

\subsection{Functional correlation decay, and proofs for the theorems.} We now return to the setting of intermittent maps. Given an admissible sequence of maps \((T_n)_{n\ge 1}\),  a result from \cite{aimino2015} guarantees that correlations decay polynomially with respect to any measure $\mu$ with density in $\cC_*(\beta_*)$. Next we recall a generalization of this result, established in \cite{leppanen2017}, which facilitates controlling integrals such as those appearing in conditions (A2) and (B2) of the previous two theorems.\\
\indent Given a function \(F: \, [0,1]^d \to \bR\), we denote by \(\text{Lip}(F;i)\) the quantity
\begin{align*}
\sup_{y_1,\ldots,y_d \in [0,1]}\sup_{a_i \neq b_i} \frac{|F(y_1,\ldots, y_{i-1}, a_i, y_{i+1},\ldots , y_d) - F(y_1,\ldots, y_{i-1}, b_i, y_{i+1},\ldots , y_d)|}{|a_i-b_i|},
\end{align*}
and say that \(F\) is Lipschitz continuous in the coordinate \(x_\alpha\), if \(\text{Lip}(F;\alpha) < \infty\).

\begin{thm}\label{thm:corr} Let \((T_n)_{n\ge 1}\) be an admissible sequence of maps. Let $k \ge 1$ be an integer, \(F: \, [0,1]^{k+1} \to \bR\) be a bounded function, and fix integers \(0  \le n_1 \le \ldots \le n_k\), \(1 \le l_1 < \ldots < l_p < k\). Suppose that \(F\) is Lipschitz continuous in the coordinates \(x_{\alpha}\), where \(1 \le \alpha \le l_p +1\),  and denote by \(H(x_0,\ldots, x_{p})\) the function
\begin{align*}
F(x_0,\widetilde{T}_{n_1}(x_0),\ldots, \widetilde{T}_{n_{l_1}}(x_0), \widetilde{T}_{n_{{l_1}+1}}(x_1),\ldots,\widetilde{T}_{n_{{l_2}}}(x_1), \ldots, \widetilde{T}_{n_{l_p+1}}(x_{p}),\ldots \widetilde{T}_{n_k}(x_{p})).
\end{align*}
Then, for any probability measures \(\mu,\mu_1,\ldots,\mu_p \) whose densities belong to \(\cC_*(\beta_*)\),
\begin{align*}
&\left| \int H(x,\ldots,x) \, d\mu(x) - \idotsint H(x_0,\ldots,x_{p}) \, d\mu(x_0)d\mu_1(x_1)\ldots \, d\mu_p(x_{p}) \right| \\\ 
&\le C(\Vert F \Vert_{\infty} + \max_{1 \le\alpha \le l_p+1} \textnormal{Lip}(F;\alpha))\sum_{i=1}^p \rho(n_{l_{i}+1}-n_{l_i}), 
\end{align*}
where \(\rho(n) = n^{-\frac{1}{\beta_*}+1}(\log n)^{\frac{1}{\beta_*}}\) for \(n \ge 2\), \(\rho(0) = \rho(1) = 1\), and \(C > 0\) is a constant depending only on \(\beta_*\).\\
\end{thm}

\begin{proof}[Proof for Theorem~\ref{thm:int_multi}] It suffices to verify conditions (A1) and (A2) of Theorem \ref{thm:main}. We do this by applying Theorem \ref{thm:corr} . 

\begin{itemize}

\item[(A1)]{ We let \(\rho(n) = n^{-\frac{1}{\beta_*}+1}(\log n)^{\frac{1}{\beta_*}}\) for \(n \ge 2\), and \(\rho(0) = \rho(1) = 1\). Then, under the standing assumption \(\beta_* < 1/3\), we have \( \sum_{i=1}^{\infty} i\rho(i) < \infty\). If \(F: \, \bR^{n+1} \to \bR\), \(F(x_0,\ldots,x_{n}) = \bar{f}_{a_1}(x_1)\cdots \bar{f}_{a_{n}}(x_{n})\), where \(a_i \in \{1,\ldots ,d\}\), then \(\Vert F \Vert_{\infty} \le 2^n\Vert f \Vert_{\infty}^n \le 2^n\Vert f \Vert_{\text{Lip}}^n \) and \( \text{Lip}(F; \alpha) \le \Vert f \Vert_{\text{Lip}}^n\) for all \(\alpha =1,\ldots, n+1\). By Theorem \ref{thm:corr},
\begin{align*}
|\mu(\bar{f}_{\alpha}^i \bar{f}^j_{\beta})| 
&\le C(\beta_*)\Vert f \Vert_{\text{Lip}}^2 \rho(i-j), \\
|\mu(\bar{f}_{\alpha}^i \bar{f}_{\beta}^j \bar{f}_{\gamma}^k \bar{f}_{\delta}^l)| & \le  C(\beta_*)\Vert f \Vert_{\text{Lip}}^4 \min\{\rho(j-i),\rho(l-k)\},
\\
|\mu(\bar{f}_{\alpha}^i \bar{f}_{\beta}^j \bar{f}_{\gamma}^k \bar{f}_{\delta}^l)) - \mu(\bar{f}_{\alpha}^i \bar{f}_{\beta}^j)\mu(\bar{f}_{\gamma}^k \bar{f}_{\delta}^l)| & \le C(\beta_*) \Vert f \Vert_{\text{Lip}}^4\,\rho(k-j),
\end{align*} 
for some constant \(C(\beta_*) > 0\) depending only on the system $T_{\beta_*}$, whenever $0\le i \le j \le k \le l < N$, and $\alpha,\beta,\gamma,\delta\in\{1,\ldots ,d\}$ .} \smallskip
\item[(A2)]{For $0\le n < N$, $0\le t\le 1$ and $v\in\bR^d$, we have by Theorem \ref{thm:corr} the upper bound
\begin{align}\label{bound_a2}
|\mu( \bar{f}^n \cdot \nabla h(v + W^n t) ) | \le C(\beta_*) d^{2}(\Vert \nabla h \Vert_{\infty} \Vert f \Vert_{\text{Lip}} + \Vert f \Vert_{\text{Lip}}^2 \Vert D^2h \Vert_{\infty})\rho(K).
\end{align}
To see this, define
\begin{align*}
F(x_0,\ldots, x_{n-K},x_n, x_{n+K}, \ldots, x_{N-1}) = \bar{f}(x_n) \cdot \nabla h \left(v + \frac{1}{\sqrt{N}} \sum_{i \notin [n]_K} \bar{f}(x_i)t \right).
\end{align*}
Then \( \Vert F \Vert_{\infty} \le  2d\Vert f \Vert_{\infty} \Vert \nabla h \Vert_{\infty}\), and  for all \(\alpha \le N\),
\begin{align*}
\text{Lip}(F; \alpha) &\le d \Vert \nabla h \Vert_{\infty} \Vert f \Vert_{\text{Lip}} + d^{2} \Vert f \Vert_{\text{Lip}}^2 N^{-\frac12}\Vert D^2h \Vert_{\infty},
\end{align*}
so that Theorem \ref{thm:corr} is applicable with \(F\). The upper bound \eqref{bound_a2} now follows.}
\end{itemize}

By Theorem \ref{thm:main}, there is a constant \(C > 0\) independent of \(N\) such that
\begin{align*}
&|\mu(h(W)) - \Phi_{\Sigma}(h)| 
\le C  \! \left(\frac {K}{\sqrt{N}} + \sum_{i= K+1}^{\infty}  \rho(i)  + \sqrt{N}\rho(K) \right).
\end{align*}
For \(K \le \sqrt{N}\), \(\sum_{i \ge K +1} i^{1-\frac{1}{\beta_*}} = \cO(K^{2-\frac{1}{\beta_*}}) = \cO(\sqrt{N} K^{1-\frac{1}{\beta_*}}) \). Hence, we choose \(K = \round{N^{\beta_*}}\) so that \(\sqrt{N} K^{1-\frac{1}{\beta_*}} \approx K/\sqrt{N} \), and
\begin{align*}
&|\mu(h(W)) - \Phi_{\Sigma}(h)| \le C N^{\beta_* - \frac12}(\log N)^{\frac{1}{\beta_*}}.
\end{align*}
The proof for Theorem \ref{thm:int_multi} is complete.

\end{proof}

\begin{proof}[Proof for Theorem~\ref{thm:int}] The proof is almost the same as the previous one: using Theorem \ref{thm:corr}, one verifies conditions (B1) and (B2) of Theorem \ref{main thm 1d}, taking \(\rho(n) = n^{-\frac{1}{\beta_*}+1}(\log n)^{\frac{1}{\beta_*}}\) for \(n \ge 2\) and \(\rho(0) = \rho(1) = 1\), and $\widetilde{\rho} = \rho$. We omit the details.
\end{proof}

\section{Proofs for Theorems \ref{thm:qds_abs_1} and \ref{thm:qds_abs_2}}\label{more proofs}

We start with a couple of simple observations. For any \(\alpha \in \{1,\ldots,d\}\) and \(t \in [0,1]\), we denote \(\hat{f}_{\alpha,t} = f_{\alpha} - \hat{\mu}_t(f_{\alpha})\). 

\begin{lem}\label{lem:old_conditions} Let \((\mathbf{T},\gamma)\) be a QDS that satisfies conditions (I) and (II). Then, the following conditions hold for all functions \(F\) of the form \(F = f_{a} \cdot f_b^q \circ T_k \circ \cdots \circ T_1\) where \(T_i \in \mathbf{T}'\), \(a,b \in \{1,\ldots,d\}\), \(q \in \{0,1\}\).
\begin{itemize}
\item[(f1)] Whenever \(n,k,l\) are integers with \(0 \le l,k \le n\) and \(|k-l| \ge n_1\), 
\begin{align*}
|\mu(\bar{f}^{n,k}_{\alpha}\bar{f}^{n,l}_{\beta})| \le C|k-l|^{-\theta_1},
\end{align*}
where \(\theta_1 = \varphi\). Moreover, for all \(s \in [0,1]\) and \(k \ge n_1\),
\begin{align*}
|\hat{\mu}_s(\hat{f}_{\alpha,s} \hat{f}_{\beta,s} \circ \gamma_s^k) | \le C k^{-\theta_1}.
\end{align*}
\smallskip
\item[(f2)] For all \(r,s \in [0,1]\), \begin{align*}
|\hat{\mu}_r(F) - \hat{\mu}_s(F) | \le C|r-s|^{\theta_2},
\end{align*} 
where
\begin{align*}
\theta_2 = \frac{\varphi \psi}{1+\varphi}.
\end{align*}
\smallskip
\item[(f3)] Set 
\begin{align*}
\theta_3 = \frac{\varphi \psi}{\varphi + \psi + 1} \hspace{0.4cm} \text{and} \hspace*{0.4cm} \eta_3 = \frac{\psi}{\varphi + \psi + 1} < \frac13.
\end{align*}
Then, there is \(c_0 > 0\), such that for all \(s \in [0,1]\) with \( c_0 n^{\eta_3-1} < s \le 1\),  
\begin{align*}
| \mu_{n\round{ns}}(F) - \hat{\mu}_{s}(F) | \le C n^{-\theta_3}.
\end{align*}
\end{itemize}
 \end{lem}
 
\begin{remark} We have
\begin{align*}
\min \{\theta_1 -1 ,\theta_2,\theta_3,\psi \} = \min \{\theta_1 -1, \theta_3\}.
\end{align*}
\end{remark}

\begin{proof}[Proof for Lemma \ref{lem:old_conditions}] (f1): This is a direct consequence of condition (I).  \\
\indent (f2): Assume that \(r \neq s\). By the second part of condition (II), for all \(n \ge 0\),
\begin{align*}
| (\gamma_r^n)_*\hat{\mu}_r (F) -  (\gamma_s^n)_*\hat{\mu}_r (F)| \le C n |r-s|^{\psi}.
\end{align*}
Hence,
\begin{align*}
&|\hat{\mu}_r(F) - \hat{\mu}_s(F) | \le | (\gamma_s^n)_*\hat{\mu}_r (F) -  (\gamma_s^n)_*\hat{\mu}_s (F)| + C n |r-s|^{\psi}.
\end{align*}
On the other hand, an application of condition (I) with \(p=0\) yields
\begin{align*}
&| (\gamma_s^n)_*\hat{\mu}_r (F) -  (\gamma_s^n)_*\hat{\mu}_s (F)| \le C n^{-\varphi}.
\end{align*}
Condition (f2) now follows by choosing \(n = \round{|r-s|^{-\frac{\psi}{1+\varphi}}}\). \\
\indent (f3): Let \(K \in \bN\) be such that \(\round{ns} \ge K \). Since \(\mu_{n,\round{ns}-K}, \hat{\mu}_s \in \cC\), condition (I) implies the upper bound
\begin{align*}
&|\mu_{n,\round{ns}-K} (F \circ T_{n,\round{ns}} \circ \cdots \circ T_{n,\round{ns}-K+1}) - \hat{\mu}_s (F \circ T_{n,\round{ns}} \circ \cdots \circ T_{n,\round{ns}-K+1}) | \\
&\le C K^{-\varphi}.
\end{align*}
On the other hand, the two bounds of condition (II) yield 
\begin{align*}
&| \hat{\mu}_s (F \circ T_{n,\round{ns}} \circ \cdots \circ T_{n,\round{ns}-K+1}) - \hat{\mu}_s (F \circ \gamma^K_s)| \\
&\le |\hat{\mu}_s (F \circ T_{n,\round{ns}} \circ \cdots \circ T_{n,\round{ns}-K+1}) - \hat{\mu}_s (F \circ \gamma_{\round{ns}/n} \circ \cdots \circ \gamma_{(\round{ns}-K+1)/n}) |\\
&+ |\hat{\mu}_s (F \circ \gamma_{\round{ns}/n} \circ \cdots \circ \gamma_{(\round{ns}-K+1)/n}) - \hat{\mu}_s (F \circ \gamma^K_s)| \\
&\le C( K n^{-\psi} + K(K/n)^{\psi}) \le Cn^{-\psi} K^{1 + \psi}.
\end{align*}
Hence,
\begin{align*}
&| \mu_{n\round{ns}}(F) - \hat{\mu}_{s}(F) | \\
&\le C K^{-\varphi} 
+ | \hat{\mu}_s (F \circ T_{n,\round{ns}} \circ \cdots \circ T_{n,\round{ns}-K+1}) - \hat{\mu}_s (F \circ \gamma^K_s)| \\
&\le C (K^{-\varphi} + n^{-\psi} K^{1 + \psi}).
\end{align*}
Condition (f3) then follows by choosing \(K = \round{n^{\frac{\psi}{\varphi + \psi + 1}}} \) . 
\end{proof}

\indent Recall that
\begin{align*}
\hat{\Sigma}_t(f)=\lim_{m\rightarrow \infty}\hat{\mu}_{t}\left[\frac{1}{\sqrt{m}}\sum_{k=0}^{m-1}\hat{f}_{t}\circ \gamma_{t}^{k}\otimes \frac{1}{\sqrt{m}}\sum_{k=0}^{m-1}\hat{f}_{t}\circ \gamma_{t}^{k}\right].
\end{align*}

\begin{lem}\label{lem:limitvar} Suppose conditions (I) and (II) hold. Then
\begin{align}\label{univ var2}
(\hat{\Sigma}_t)_{\alpha\beta}(f) = \hat{\mu}_t[\hat{f}_{\alpha,t} \hat{f}_{\beta,t}] + \sum_{k=1}^{\infty} \hat{\mu}_t[\hat{f}_{\alpha,t} \hat{f}_{\beta,t} \circ \gamma_t^k+\hat{f}_{\beta,t} \hat{f}_{\alpha,t} \circ \gamma_t^k],
\end{align}
 for all \(t \in [0,1]\), \(\alpha,\beta \in \{1,\ldots,d \}\). Moreover, for each \(\alpha,\beta \in \{1,\ldots,d\} \) the map \(t \mapsto (\hat{\Sigma}_t)_{\alpha\beta}(f)\) is Hölder continuous  with exponent $\psi (\varphi - 1)/(\varphi + 1)$.
\end{lem}

\begin{proof} 
By definition
\beq
(\hat{\Sigma}_t)_{\alpha\beta}(f)= \lim_{m\rightarrow \infty} \frac1m \hat{\mu}_{t}\left[\sum_{k=0}^{m-1}\hat{f}_{\alpha,t} \circ \gamma_{t}^{k}\sum_{k=0}^{m-1}\hat{f}_{\beta,t}\circ \gamma_{t}^{k}\right].\label{eq:cov matr}
\eeq

A straightforward manipulation of \eqref{eq:cov matr} shows that 
\begin{align*}
(\hat{\Sigma}_t)_{\alpha\beta}(f) &= \hat{\mu}_t[\hat{f}_{\alpha,t}\hat{f}_{\beta,t}] +  \lim_{m \to \infty} \frac1m   \sum_{k=1}^{m-1} (m-k) \hat{\mu}_t[\hat{f}_{\alpha,t}\hat{f}_{\beta,t} \circ \gamma_{t}^{k}+ \hat{f}_{\beta,t} \hat{f}_{\alpha,t}  \circ \gamma_{t}^{k}]
\\
&= \hat{\mu}_t[\hat{f}_{\alpha,t}  \hat{f}_{\beta,t} ] +  \lim_{m \to \infty}  \sum_{k=1}^{m-1} \hat{\mu}_t[\hat{f}_{\alpha,t}  \hat{f}_{\beta,t}  \circ \gamma_{t}^{k} + \hat{f}_{\beta,t} \hat{f}_{\alpha,t}  \circ \gamma_{t}^{k}] 
\\
&- \lim_{m \to \infty} \frac1m   \sum_{k=1}^{m-1}k \hat{\mu}_t[\hat{f}_{\alpha,t} \hat{f}_{\beta,t} \circ \gamma_{t}^{k}+\hat{f}_{\beta,t}  \hat{f}_{\alpha,t}  \circ \gamma_{t}^{k}],
\end{align*}
where the limits exist by condition (f1). By another  application of condition (f1), we see that
\begin{align*}
\lim_{m \to \infty} \frac1m  \sum_{k=1}^{m-1} k| \hat{\mu}_t[\hat{f}_{\alpha,t}  \hat{f}_{\beta,t}  \circ \gamma_{t}^{k}+\hat{f}_{\beta,t}  \hat{f}_{\alpha,t}  \circ \gamma_{t}^{k}]| = 0,
\end{align*}
so that \eqref{univ var2} follows.\\
\indent Next we study the continuity of $\hat{\Sigma}_{t}(f)$. Let $\alpha,\beta\in\{1,...,d\}$. Conditions (f2) and (II) imply
\begin{align*}
\hat{\mu}_t[\hat{f}_{\alpha,t}  \hat{f}_{\beta,t}  \circ \gamma_t^k + \hat{f}_{\beta,t}  \hat{f}_{\alpha,t}  \circ \gamma_t^k] = \hat{\mu}_s[\hat{f}_{\alpha,s}  \hat{f}_{\beta,s}  \circ \gamma_s^k + \hat{f}_{\beta,s}  \hat{f}_{\alpha,s}  \circ \gamma_s^k] + \cO(|t-s|^{\theta_2} + k|t-s|^{\psi}).
\end{align*}
Hence,
\begin{align*}
|(\hat{\Sigma}_{t})_{\alpha\beta}(f) -(\hat{\Sigma}_{s})_{\alpha\beta}(f)|
& \le \sum_{k=0}^{K-1}| \hat{\mu}_t[\hat{f}_{\alpha,t}  \hat{f}_{\beta,t}  \circ \gamma_t^k + \hat{f}_{\beta,t}  \hat{f}_{\alpha,t}  \circ \gamma_t^k]- \hat{\mu}_s[\hat{f}_{\alpha,s}  \hat{f}_{\beta,s}  \circ \gamma_s^k+ \hat{f}_{\beta,s} \hat{f}_{\alpha,s}  \circ \gamma_s^k]|
\\
 &+ \sum_{k=K}^{\infty}|\hat{\mu}_t[\hat{f}_{\alpha,t}  \hat{f}_{\beta,t}  \circ \gamma_t^k+ \hat{f}_{\beta,t}  \hat{f}_{\alpha,t} \circ \gamma_t^k]- \hat{\mu}_s[\hat{f}_{\alpha,s}  \hat{f}_{\beta,s}   \circ \gamma_s^k+ \hat{f}_{\beta,s}   \hat{f}_{\alpha,s}  \circ \gamma_s^k]|
\\
&\le C\sum_{k=0}^{K-1}( k|t-s|^{\psi} + |t-s|^{\theta_2} ) +C\sum_{k=K}^{\infty}k^{-\varphi} \\
&\le C( K^{2}|t-s|^{\psi} + K|t-s|^{\theta_2} + K^{1-\varphi}),   
\end{align*}
where condition (f1) was used in the second last inequality and \(C > 0\) does not depend on ~\(t\) or ~\(s\). The latter claim of the lemma now follows by choosing $K=\round{|t-s|^{-\psi/(\varphi + 1) }}$.
\end{proof}

\subsection{Proof for Theorem \ref{thm:qds_abs_1}}
We first prove the following lemma.

\begin{lem}\label{secmoment_d} Assume conditions (I) and (II). Then, fixing any \(0 < \kappa < \eta_3 < \delta < \tfrac12\),
\begin{align*}
 &\mu[(\xi_n(t+h) - \xi_n(t))_{\alpha} (\xi_n(t+h) - \xi_n(t))_{\beta}] \\
 &=   \int_{t}^{t + h} [\hat{\Sigma}_s(f)]_{\alpha\beta} \,ds + h\cO(n^{\kappa (1-\varphi)}) + \cO(n^{-1+\delta+\kappa}).
\end{align*}
holds whenever whenever \(0 \le t \le t + h \le 1\) and \(\alpha, \beta \in \{1,\ldots, d\}\). Here the error term is uniform in \(t\) and \(h\).
  \end{lem}

\begin{proof} Let \(h > 0\) and let \(n\) be sufficiently large so that \(n^{-1/2} \le h\). Let \(\eta_3 < 1/3\) and ~\(c_0\) be as in condition (f3) of Theorem \ref{lem:old_conditions}. We fix numbers \(\kappa, \delta \in (0,1)\) such that \(0 < \kappa < \eta_3 < \delta < \frac12 \),
 and denote \(a_n = n^{-1+ \kappa}\), and \(b_n = n^{-1+\delta}\). Then, we can partition \([t,t+h]^2 = P_n \cup Q_n \cup R_n \ \), where
\begin{align*}
P_n = \{ (s,r) \in [t,t+h]^2 \: : \: t + b_n \le s \le t + h - b_n \text{ and } |r-s| \le a_n \},
\end{align*}
\begin{align*}
Q_n = \{ (s,r) \in [t,t+h]^2 \: : \: |r-s| \le a_n \text{ and either } s < t + b_n \text{ or } s > t + h - b_n \},
\end{align*}
and
\begin{align*}
R_n = \{ (s,r) \in [t,t+h]^2 \: : \: |r-s| > a_n \}.
\end{align*}
Since \(m(Q_n) = \cO(a_nb_n)\),
\begin{align*}
\left| n\iint_{Q_n} \mu(\bar{f}^{n,\round{ns}}_{\alpha}\bar{f}^{n,\round{nr}}_{\beta}) \, dr \,ds \right| \le 4C \Vert f \Vert_{\infty}^2 na_nb_n \le  4C \Vert f \Vert_{\infty}^2 n^{-1 + \delta + \kappa}.
\end{align*}
On the other hand, when \(n^{\kappa} = na_n \ge n_1\), it follows by condition (f1) that
\begin{align*}
&\left| n \iint_{R_n} \mu(\bar{f}^{n,\round{ns}}_{\alpha} \bar{f}^{n,\round{nr}}_{\beta}) \, dr \,ds \right| \\
&= \left| n \int_t^{t+h-a_n} \int_{s+ a_n}^{t+h} \mu(\bar{f}^{n,\round{ns}}_{\alpha} \bar{f}^{n,\round{nr}}_{\beta}) + \mu(\bar{f}^{n,\round{ns}}_{\beta} \bar{f}^{n,\round{nr}}_{\alpha}) \, dr \,ds \right| \\
&\le  2n C \int_t^{t+h-a_n} \int_{s+ a_n}^{t+h} (nr-ns)^{-\theta_1} \, dr \, ds  
\le 2hC n^{1-\theta_1} \frac{1}{\theta_1 -1} a_n^{1-\theta_1} 
= \frac{2hC}{\theta_1 -1} n^{\kappa (1-\theta_1)}.
\end{align*}
Hence, only the contribution from the diagonal \(P_n\) is significant as \(n \to \infty\):
\begin{align*}
& n \iint_{[t,t+h]^2} \mu(\bar{f}^{n,\round{ns}}_{\alpha} \bar{f}^{n,\round{nr}}_{\beta}) \, dr \,ds = n \iint_{P_n} \mu(\bar{f}^{n,\round{ns}}_{\alpha} \bar{f}^{n,\round{nr}}_{\beta}) \, dr \,ds +  \cO( hn^{\kappa (1-\theta_1)} +  n^{-1 + \delta + \kappa}),
\end{align*}
where the error is uniform in \(t\) and \(h\). For all \((s,r) \in P_n\), we have the lower bound
\begin{align*} r &\ge s - a_n \ge t + b_n - a_n \ge b_n - a_n = (n^{\delta-\eta_3} - n^{\kappa - \eta_3})n^{-1+\eta_3} \\
&\ge (n^{\delta-\eta_3} - 1)n^{-1+\eta_3} > c_0n^{-1+\eta_3},
\end{align*} 
when \(n^{\delta-\eta_3} > 1 + c_0\). For such \(n\) and for all \(r \in (s-a_n, s+ a_n)\), conditions (f2) and (f3) imply the bound
\begin{align*}
& | \mu(f^{n\round{nr}}_{\alpha}) - \hat{\mu}_{s}(f_{\alpha}) | 
\le C ( n^{-\theta_3} +  |r-s|^{\theta_2}) 
\le C( n^{-\theta_3} + n^{(\kappa -1)\theta_2}),
\end{align*}
which continues to hold if \(\beta\) is substituted for \(\alpha\). It follows that
\begin{align*}
& n \int_{s - a_n}^{s+a_n} \mu(\bar{f}^{n,\round{ns}}_{\alpha}\bar{f}^{n,\round{nr}}_{\beta}) \, dr \\
&= n \int_{s - a_n}^{s+a_n} \mu(f^{n,\round{ns}}_{\alpha} f^{n,\round{nr}}_{\beta}) - \hat{\mu}_s(f_{\alpha})\hat{\mu}_s(f_{\beta}) \, dr  
+ \cO (n^{\kappa - \theta_3} + n^{\kappa + (\kappa-1)\theta_2}).
\end{align*}
Next we split the domain of integration \([s-a_n, s+a_n] = [s-a_n,s] \cup [s,s+a_n] \), and consider the right half. Denoting \(c_n = \frac1n (1 - \{ns\}) \), we have
\begin{align*}
&n \int_s^{s+a_n} \mu(f^{n,\round{ns}}_{\alpha} f^{n,\round{nr}}_{\beta}) \, dr = n \int_0^{a_n} \mu(f^{n,\round{ns}}_{\alpha} f^{n,\round{n(s+r)}}_{\beta}) \, dr \\
&= c_n n \mu_{n,\round{ns}}(f_{\alpha}f_{\beta}) + n \int_{c_n}^{a_n} \mu_{n,\round{ns}}(f_{\alpha} f_{\beta}\circ T_{n,\round{n(s+r)}} \circ \cdots \circ T_{n,\round{ns} +1} ) \, dr.
\end{align*}
By condition (f3),
\begin{align*}
 c_n n \mu_{n,\round{ns}}(f_{\alpha}f_{\beta}) = c_n n \hat{\mu}_s(f_{\alpha}f_{\beta}) + \cO(n^{-\theta_3}).
\end{align*}
Another application of condition (f3) and two applications of condition (II) yield
\begin{align*}
&n \int_{c_n}^{a_n} \mu_{n,\round{ns}}(f_{\alpha} f_{\beta}\circ T_{n,\round{n(s+r)}} \circ \cdots \circ T_{n,\round{ns} +1} ) \, dr \\
&= n \int_{c_n}^{a_n} \mu_{n,\round{ns}}(f_{\alpha} f_{\beta} \circ \gamma_{\round{n(s+r)}/n} \circ \cdots \circ \gamma_{(\round{ns}+1)/n} ) \, dr + n \int_{c_n}^{a_n} \cO( r n^{1-\psi} ) \, dr \\
&= n \int_{c_n}^{a_n} \hat{\mu}_s( f_{\alpha} f_{\beta} \circ \gamma_{\round{n(s+r)}/n} \circ \cdots \circ \gamma_{(\round{ns}+1)/n} ) \, dr   + \cO (na_n n^{-\theta_3}) + \cO(n^{2\kappa - \psi}) \\
&= n \int_{c_n}^{a_n} \hat{\mu}_s(f_{\alpha} f_{\beta} \circ \gamma_{s}^{\round{n(s+r)}-\round{ns}}) \, dr   + n^2 \int_{c_n}^{a_n} \cO(r^{1+\psi}) \, dr + \cO ( n^{\kappa -\theta_3}) + \cO(n^{ 2\kappa - \psi}) \\
&= n \int_{c_n}^{a_n} \hat{\mu}_s(f_{\alpha} f_{\beta} \circ \gamma_{s}^{\round{n(s+r)}-\round{ns}}) \, dr   + \cO (n^{- \psi + \kappa (2+ \psi)}) + \cO ( n^{\kappa -\theta_3}).
\end{align*}
Hence,
\begin{align*}
&n \int_s^{s+a_n} \mu(f^{n,\round{ns}}_{\alpha} f^{n,\round{nr}}_{\beta}) \, dr \\
&= n \int_{0}^{a_n} \hat{\mu}_{s}(f_{\alpha} f_{\beta}\circ \gamma_s^{\round{n(s+r)} - \round{ns}})  \, dr + \cO (n^{- \psi + \kappa (2+ \psi)}) + \cO ( n^{\kappa -\psi}).
\end{align*}
A similar argument shows
\begin{align*}
 n \int_{s-a_n}^{s} \mu(f_{n,\round{ns}} f_{n,\round{nr}}) \, dr = n \int_{-a_n}^{0} \hat{\mu}_{s}(f_{\alpha} f_{\beta}\circ \gamma_s^{-\round{n(s+r)} + \round{ns}}) \, dr + \cO (n^{- \psi + \kappa (2+ \psi)}) + \cO ( n^{\kappa -\theta_3}).
\end{align*}
Since
\begin{align*}
[\hat{\Sigma}_s(f)]_{\alpha\beta} = n \int_{- \infty}^{\infty} \hat{\mu}_{s}(\hat{f}_{\alpha,s} \hat{f}_{\beta,s}  \circ \gamma_s^{|\round{n(s+r)} - \round{ns}|})\, dr,
\end{align*}
the foregoing estimates together with condition (f1) imply
\begin{align*}
&n \int_{s - a_n}^{s+a_n} \mu(\bar{f}^{n,\round{ns}}_{\alpha}\bar{f}^{n,\round{nr}}_{\beta}) \, dr \\
&= n \int_{- a_n}^{a_n} \hat{\mu}_{s}(\hat{f}_{\alpha,s}  \hat{f}_{\beta,s} \circ \gamma_s^{|\round{n(s+r)} - \round{ns}|})\, dr + \cO (n^{- \psi + \kappa (2+ \psi)})  +  \cO (n^{\kappa - \theta_3} + n^{\kappa + (\kappa-1)\theta_2}) \\
&= n \int_{- \infty}^{\infty} \hat{\mu}_{s}(\hat{f}_{\alpha,s} \hat{f}_{\beta,s}  \circ \gamma_s^{|\round{n(s+r)} - \round{ns}|})\, dr  + \cO(n^{\kappa (1-\theta_1)}) + \cO (n^{- \psi + \kappa (2+ \psi)} + n^{\kappa - \theta_3} + n^{\kappa + (\kappa-1)\theta_2}).
\end{align*}
Since \(\kappa < \eta_3 = \psi/(\varphi + \psi + 1)\), it follows that 
\begin{align*}
\kappa (1-\theta_1) = \max \{\kappa (1-\theta_1), -\psi + \kappa (2 + \psi), \kappa - \theta_3, \kappa + (\kappa-1)\theta_2\}.
\end{align*}
We conclude that given any \(0 < \kappa < \eta_3 < \delta < \frac12 \), 
\begin{align*}
&n \int_t^{t+h} \int_t^{t+h} \mu(\bar{f}^{n,\round{ns}}_{\alpha} \bar{f}^{n,\round{nr}}_{\beta}) \, dr \,ds \\
&=n \int_{t+b_n}^{t + h - b_n}\int_{s-a_n}^{s+a_n} \mu(\bar{f}^{n,\round{ns}}_{\alpha} \bar{f}^{n,\round{nr}}_{\beta}) \, dr \,ds +  \cO( hn^{\kappa (1-\theta_1)} +  n^{-1 + \delta + \kappa}) \\
&= \int_{t}^{t + h} [\hat{\Sigma}_s(f)]_{\alpha\beta} \,ds + h\cO(n^{\kappa (1-\theta_1)}) + \cO(n^{-1+\delta+\kappa}).
\end{align*}
This completes the proof for Lemma \ref{secmoment_d}.
\end{proof}

Recall that $\Sigma_{n,t} = \mu [\xi_n(t) \otimes \xi_n(t)]$. In Lemma \ref{secmoment_d}, replacing $t$ with $0$ and $h$ with $t$ yields
\beqn
[\Sigma_{n,t}]_{\alpha\beta}=\mu[(\xi_n(t))_{\alpha} (\xi_n(t))_{\beta}] =  [\Sigma_{t}]_{\alpha\beta}  +  \cO (tn^{\kappa(1- \varphi)} + n^{-1 + \delta + \kappa}).\label{covariance difference}
\eeqn

This holds especially, when $tn\in \bN$, $i.e.$, $t=\lceil nt \rceil/n$. Let $0<\epsilon< \max\{1/6,\eta_{3}(\varphi-1)/2 \}$. Recall that $\eta_{3}=\psi/(\varphi + \psi + 1) < 1/3$. Choose $\kappa=\eta_{3}+\epsilon(1-\varphi)^{-1}$ and $\delta= 1/2-\epsilon$. We leave it to the reader to check that with these choices $0 < \kappa < \eta_3 < \delta < \tfrac12$. Now applying Lemma \ref{secmoment_d} gives  $[\Sigma_{n,t}]_{\alpha\beta}- [\Sigma_{t}]_{\alpha\beta} =\cO (tn^{\kappa(1- \varphi)} + n^{-1 + \delta + \kappa})= \cO (n^{\eta_{3}(1- \varphi)+\epsilon} + n^{-1 + 1/2-\epsilon + \eta_{3}+\epsilon(1-\varphi)^{-1}})= \cO (n^{\eta_{3}(1- \varphi)+\epsilon} + n^{-1/6})=\cO (n^{(\psi-\varphi\psi)/(\varphi + \psi + 1)+\epsilon} + n^{-1/6})$. 

Therefore for all $\alpha,\beta\in \{1,2,...,d \}$ it holds
$
[\Sigma_{n,t}]_{\alpha\beta}-[\Sigma_{t}]_{\alpha\beta} \le \cO(n^{\max\{(\psi-\varphi\psi)/(\varphi + \psi + 1)+\epsilon,-1/6\}}).
$
This completes the proof for Theorem \ref{thm:qds_abs_1}.

\subsection{Proof for Theorem \ref{thm:qds_abs_2}}

Let \(h: \, \bR^d \to \bR\) be Lipschitz, $\epsilon>0$ and
assume that conditions (I)--(IV) hold.

Let $n\in \bN$ be arbitrary and $t\ge t_{0}$. 
As in \cite{Hella_2018}, we split $\mu\left[h(\xi_{n}(t))\right]-\Phi_{\Sigma_{t}}(h)$ into four terms, whose absolute values can be controlled by a function of $n$:
\beq
\left|\mu\left[h(\xi_{n}(t))\right]- \mu\left[h\left(\xi_{n}\left(\lceil nt\rceil/n\right)\right)\right] \right|,\label{eq:multiv 1}
\eeq

\beq
\left|\mu\left[h\left(\xi_{n}\left(\lceil nt\rceil/n\right)\right)\right] - \Phi_{\Sigma_{n,\lceil nt\rceil/n}}(h)\right|,\label{eq:multiv 2}
\eeq

\beq
\left|\Phi_{\Sigma_{n,\lceil nt\rceil/n}}(h)-\Phi_{\Sigma_{\lceil nt\rceil/n}}(h)\right|,\label{eq:multiv 3}
\eeq
and 
\beq
\left|\Phi_{\Sigma_{\lceil nt\rceil/n}}(h)-\Phi_{\Sigma_{t}}(h)\right|.\label{eq:multiv 4}
\eeq

\textbf{Term \eqref{eq:multiv 1}.} As in \cite{Hella_2018}, the boundedness of $f$ implies the uniform bound
$\eqref{eq:multiv 1}\le Cn^{-\frac{1}{2}}$.

\textbf{Term \eqref{eq:multiv 2}.}
 By condition (III),
\beqn
 \left|\mu\left[h(\xi_{n}(\lceil nt\rceil/n))\right]- \Phi_{\Sigma_{n,\lceil nt\rceil/n)}}(h)\right|  \le Cn^{-\zeta}. 
\eeqn
\indent\textbf{Term \eqref{eq:multiv 3}.}
Let $Z\sim \cN(0,I_{d\times d})$, where $\cN(0,I_{d\times d})$ is a standard $d$-dimensional normal distribution. Given a matrix $A \in \bR^{d\times d}$, we let $\lambda_{1}(A)$ denote the smallest eiqenvalue of $A$, and set $|A| = \max_{1 \le \alpha, \beta \le d} |A_{\alpha\beta}|$. If $\Sigma_{1}$ and $\Sigma_{2}$ are positive definite $d\times d$ matrices, then the following bound computed in \cite{Hella_2018} holds:
\beq
\left|\Phi_{\Sigma_{1}}(h)-\Phi_{\Sigma_{2}}(h)\right|\le  \operatorname{Lip}_{d}(h)\bE|Z|_{d}\frac{ d\left|\Sigma_{1}-\Sigma_{2}\right|}{\sqrt{\lambda_1\left(\Sigma_{1}\right)}+\sqrt{\lambda_1\left(\Sigma_{2}\right)}}. \label{almost final estimate}
\eeq
Here
\beqn
\operatorname{Lip}_{d}(h)=\sup_{x,y\in\bR^d, x\neq y}\frac{|h(x)-h(y)|}{|x-y|_{d}}\quad\text{ and }\quad |x|_{d}=\sqrt{x_{1}^{2}+...+x_{d}^{2}}.
\eeqn
 
 Condition (IV) implies that  $\hat{\Sigma}_{t_{0}}$ is positive definite. For otherwise there exists a unit vector $v\in\bR^d$ such that $v\cdot \hat{\Sigma}_{t_{0}} v = 0$. Defining $f_v = v\cdot \hat{f}_t$, we have $0 = v\cdot \Sigma v = \hat{\mu}_t(f_v\,f_v) + 2\sum_{n=1}^{\infty} \hat{\mu}_t(f_v \,f_v\circ \gamma_t^n)$. It follows from \cite{Robinson_1960,Leonov_1961} that there exists a function $g_v\in L^2(\mu)$ such that $f_v = g_v - g_v\circ T$, i.e., $\hat{f}_t$ is a coboundary in the direction~$v$, which is a contradiction. 
 
\indent Since $s\mapsto \hat{\Sigma}_{s}$ is Hölder continuous it follows from the positive definiteness of $\hat{\Sigma}_{t_{0}}$  that $\hat{\Sigma}_{s}$ is positive definite for all $s$ is in some neighbourhood of $t_{0}$. Thus $\Sigma_{\lceil nt\rceil/n}=\int_{0}^{\lceil nt\rceil/n} \hat{\Sigma}_{s}ds$ is positive definite, when  $ t\ge t_{0}$. As stated in \cite{Hella_2018}, we have $\lambda_{1}(\Sigma_{t})\ge \lambda_{1}(\Sigma_{t_{0}}) $. 
 
Let $M$ be $d\times d$ matrix, we write
\beq
\|M\|=\sup_{v\in \bR^{d}\setminus \{0\}}\frac{|Mv|_{d}}{|v|_{d}}. \label{spectral norm}
\eeq
 The following fact from the same paper essentially states that small perturbations on a positive definite matrix can not change its smallest eigenvalue too much.

\textbf{Matrix fact.} Let $M$ be a positive definite matrix satisfying $\lambda_{1}(M)\ge C_{l}$ and $\|M\|\le C_{u}$ for some $0<C_{l}<C_{u}$. Then there exists $\delta= \delta(C_{l},C_{u})>0$ such that if $\tilde{M}$ is positive semi-definite and $|M-\tilde{M}|<\delta$, then $\lambda_{1}(\tilde{M})\ge C_{l}/2$ and especially $\tilde{M}$ is positive definite.

Obviously $\Sigma_{n,\lceil nt\rceil/n}$ is at least positive semidefinite, since it is a covariance matrix. There exists $C_{u}$ such that $\|\Sigma_{t}\|\le C_{u}$ for all $t\ge t_{0}$. Choose $C_{l}=\lambda_{1}(\Sigma_{t_{0}})$. The Matrix fact yields $\delta>0$ such that if $t\ge t_{0}$ and $|\Sigma_{n,t}-\Sigma_{t}| < \delta$, then $\Sigma_{n,t}$ is positive definite.  By Theorem \ref{thm:qds_abs_1} we have $|\Sigma_{n,t}-\Sigma_{t}|\le C n^{\max\{(\psi-\varphi\psi)/(\varphi + \psi + 1)+\epsilon,-1/6\}} $ uniformly. Therefore, there exists ~$n_{0}\in \bN$ such that $\Sigma_{n,t}$ is positive definite for every $t\ge t_{0}$ and $n\ge n_{0}$. 
   
Now \eqref{almost final estimate} is applicable  and Theorem \ref{thm:qds_abs_1} yields
\begin{align}
&\left|\Phi_{\Sigma_{n, \lceil nt\rceil/n}}(h)-\Phi_{\Sigma_{\lceil nt\rceil/n}}(h)\right| \notag \\
&\le  C\left|\Sigma_{n,\lceil nt\rceil/n}-\Sigma_{\lceil nt\rceil/n}\right| \le C n^{\max\{(\psi-\varphi\psi)/(\varphi + \psi + 1)+\epsilon,-1/6\}}\label{eq: n ge n_0},
\end{align}
for every $t\ge t_{0}$ and $n\ge n_{0}$. Since $\sup_{t\ge t_{0}} \left|\Phi_{\Sigma_{n, \lceil nt\rceil/n}}(h)-\Phi_{\Sigma_{\lceil nt\rceil/n}}(h)\right|$ is bounded for every ~$n\in \bN$, we can choose large enough $C$ such that \eqref{eq: n ge n_0} holds for all $n\ge 1$.

\textbf{Term \eqref{eq:multiv 4}.} We have $|\Sigma_{\lceil nt\rceil/n}- \Sigma_{t}| = | \int_{t}^{\lceil nt\rceil/n} \hat{\Sigma}_{s}ds| \le Cn^{-1} $. Note that $\Sigma_{\lceil nt\rceil/n}$ and $\Sigma_{t}$ are positive definite. Thus \eqref{almost final estimate} yields
\beq
\left|\Phi_{\Sigma_{\lceil nt\rceil/n}}(h)-\Phi_{\Sigma_{t}}(h)\right|\le Cn^{-1}.
\eeq

The result of Theorem \ref{thm:qds_abs_2} now follows from the estimates for the terms \eqref{eq:multiv 1}--\eqref{eq:multiv 4}.

\bibliography{qds}{}
\bibliographystyle{plainurl}

\end{document}